\documentclass[11pt, reqno]{amsart}  
\usepackage{amssymb,amsmath,xcolor,latexsym,enumerate,graphicx,bbm,cite,
mathtools, varwidth}
\usepackage[shortlabels]{enumitem}
\usepackage{sidecap} %
\usepackage{ifthen}
\usepackage{float}
\usepackage[%
hypertexnames=false
]{hyperref} 
\hypersetup{
    colorlinks = true,
    urlcolor = black,
    linkcolor = black,
    citecolor = black
}
\usepackage{cleveref} 
\usepackage{autonum} 
\usepackage{mleftright} 
\newtheorem{theorem}{Theorem} 
\numberwithin{equation}{section}
\newtheorem{corollary}[theorem]{Corollary}
\newtheorem{proposition}[theorem]{Proposition}
\newtheorem{lemma}[theorem]{Lemma}

\newtheorem*{exam}{Example}
\newenvironment{example}{\begin{exam}\rm}{\end{exam}}

\makeatletter 
\newtheorem*{rep@theorem}{\rep@title}\newcommand{\newreptheorem}[2]{%
\newenvironment{rep#1}[1]{%
\def\rep@title{\bf #2 \ref{##1}}%
\begin{rep@theorem}}%
{\end{rep@theorem}}}
\makeatother
\newreptheorem{theorem}{Theorem}

\theoremstyle{definition}
\newtheorem{remark}[theorem]{Remark}

\newcommand\commentout[1]{}
\newcommand\Def[1]{{\bf #1}}

\newcommand\lcm{\operatorname{lcm}} 
\newcommand\vol{\operatorname{vol}} 
\newcommand\cone{\operatorname{cone}} 
\newcommand\conv{\operatorname{conv}} 
\newcommand\Int{\operatorname{Int}} 
\newcommand\Rat{\operatorname{Rat}}

\newcommand\ehr{\operatorname{ehr}_\mathbb{Z}} 
\newcommand\Ehr{\operatorname{Ehr}_\mathbb{Z}}

\newcommand\ZZ{\mathbb{Z}}
\newcommand\QQ{\mathbb{Q}}
\newcommand\RR{\mathbb{R}}
\newcommand\CC{\mathbb{C}}

\newcommand\bA{\mathbf{A}}

\newcommand\ba{\mathbf{a}}
\newcommand\bb{\mathbf{b}}

\newcommand\bp{\mathbf{p}}
\newcommand\br{\mathbf{r}}
\newcommand\bv{\mathbf{v}}

\newcommand\bx{\mathbf{x}}
\newcommand\by{\mathbf{y}}
\newcommand\bz{\mathbf{z}}
\newcommand\bzero{\mathbf{0}}

\newcommand{\Pol}{\mathsf{P}}
\newcommand{\Qol}{\mathsf{Q}}
\newcommand{\Con}{\mathsf{C}}
\newcommand{\hstar}{\operatorname{h}^*_\ZZ}
\newcommand{\WEhr}{\operatorname{WEhr}} 
\newcommand{\htilde}{\widetilde{\operatorname{h}}} 
\newcommand{\refinedrationalEhr}[2]{\operatorname{Ehr}_\mathbb{Q}^{\sf{ref}}\mleft(#1;#2\mright)} 
\newcommand{\rationalehr}{\operatorname{ehr}_\mathbb{Q}} 
\newcommand{\rationalEhr}[2]{\operatorname{Ehr}_\mathbb{Q}\mleft(#1;#2\mright)} 
\newcommand{\refinedrationalhstar}[3]{%
      {\operatorname{h}^{*\sf{ref}}_{\mathbb{Q}}\mleft(#1;#2;#3\mright)}
}
\newcommand{\refinedrationalhstarpolynomial}{\operatorname{h}^{*\sf{ref}}_{\mathbb{Q}}} 
\newcommand{\rationalhstar}[3]{%
      {\operatorname{h}^*_{\mathbb{Q}}\mleft(#1;#2;#3\mright)}
}
\newcommand{\rationalhstarpolynomial}{\operatorname{h}^*_{\mathbb{Q}}} 
\newcommand{\realehr}{\operatorname{ehr}_\mathbb{R}} 

\newcommand{\gam}{\gamma} 

\begin{document}

\title{Rational Ehrhart Theory}

\author{Matthias Beck}
\address{Department of Mathematics, San Francisco State University}
\email{becksfsu@gmail.com}

\author{Sophia Elia}
\address{Mathematisches Institut, Freie Universit\"at Berlin}
\email{sophiae56@math.fu-berlin.de}

\author{Sophie Rehberg}
\address{Mathematisches Institut, Freie Universit\"at Berlin}
\email{s.rehberg@fu-berlin.de}

\begin{abstract}
The Ehrhart quasipolynomial of a rational polytope $\Pol$ encodes fundamental arithmetic
data of $\Pol$, namely, the number of integer lattice points in positive integral dilates of $\Pol$. Ehrhart
quasipolynomials were introduced in the 1960s, they satisfy several fundamental structural
results and have applications in many areas of mathematics and beyond. The enumerative
theory of lattice points in rational (equivalently, real) dilates of rational polytopes is
much younger, starting with work by Linke (2011), Baldoni--Berline--K\"oppe--Vergne (2013),
and Stapledon (2017). 
We introduce a generating-function \emph{ansatz} for rational Ehrhart
quasipolynomials, which unifies several known results in classical and rational Ehrhart
theory. 
In particular, we define $\gamma$-rational Gorenstein polytopes, which extend the classical notion to the rational setting and encompass the generalized reflexive polytopes studied by Fiset--Kasprzyk
(2008) and Kasprzyk--Nill (2012).
\end{abstract}



\date{04 August 2023}

\maketitle


\section{Introduction}

Let $\Pol\subseteq \RR^d$ be a $d$-dimensional \Def{lattice polytope}; that is, $\Pol$ is
the convex hull of finitely many points in $\ZZ^d$. Ehrhart's famous theorem~\cite{ehrhartpolynomial} then says that
the counting function $\ehr(\Pol;n)\coloneqq \lvert
n\Pol\cap\ZZ^d\rvert$ is a polynomial in $n\in\ZZ_{>0}$, the \Def{Ehrhart polynomial} of $\Pol$.
Equivalently, the corresponding \Def{Ehrhart series} is of the form
\begin{equation}
 \Ehr\mleft(\Pol;t\mright) \ \coloneqq \ 1 + \sum_{n\in \ZZ_{ > 0}} \ehr\mleft(\Pol;n\mright) \, t^n \ = \ \frac{\hstar(\Pol;t)}{\left(1-t\right)^{d+1}}
\end{equation}
where $\hstar(\Pol;t) \in \ZZ[t]$ is a polynomial of degree $\le d$.
Here one can consider $\Ehr\mleft(\Pol;t\mright)$ (and all series below) as a formal power series in $t$, or as an analytic power series with $|t| < 1$.

More generally, let $\Pol \subseteq \RR^d$ be a \Def{rational polytope} with \Def{denominator} $k$, i.e., $k$ is the smallest positive integer such that $k\Pol$ is a lattice polytope.
Then $\ehr(\Pol;n)$ is a \Def{quasipolynomial}, i.e., of the form
$\ehr(\Pol;n) = c_d(n) n^d + \dots + c_1(n) n + c_0(n)$ where 
$c_0, c_1,\dots, c_d\colon \ZZ\to\QQ$ are periodic functions.
The least common period of $c_0(n), c_1(n),
\dots, c_d(n)$ is the \Def{period} of $\ehr(\Pol;n)$; this period
divides the denominator $k$ of $\Pol$; again this goes back to Ehrhart~\cite{ehrhartpolynomial}. Equivalently,
\begin{equation}\label{eq:ehrgenfct}
 \Ehr\mleft(\Pol;t\mright) \ \coloneqq \ 1 +  \sum_{n\in \ZZ_{> 0}} \ehr\mleft(\Pol;n\mright) \, t^n \ 
 = \ \frac{\hstar(\Pol;t)}{(1-t^k)^{d+1}}
\end{equation}
where $\hstar(\Pol;t) \in \ZZ[t]$ has degree $< k \left(d+1\right)$.
The step from $\ehr\mleft(\Pol;n\mright)$ to $\hstar(\Pol;t)$ is essentially a change of
basis; see, e.g.,~\cite[Section~4.5]{crt}.

Because polytopes can be described by a system of linear equalities and inequalities, they
appear in a wealth of areas; likewise Ehrhart quasipolynomials have applications in number
theory, combinatorics, computational geometry, commutative algebra, representation theory,
and many other areas.
For general background on Ehrhart theory and connections to various mathematical fields, see, e.g.,~\cite{ccd}.

Our aim is to study Ehrhart counting functions with a real dilation parameter.
However, as~$\Pol$ is a rational polytope, it suffices to compute
 this counting function at certain rational arguments to fully understand it; we
will (quantify and) make this statement precise shortly (\Cref{cor:realehr} below).
We define the \Def{rational Ehrhart counting function} 
\begin{equation}
 \rationalehr\mleft(\Pol;\lambda\mright) \ \coloneqq \ \left\lvert\lambda \Pol\cap\ZZ^d\right\rvert \, ,
\end{equation}
where $\lambda\in\QQ$.
To the best of our knowledge, Linke~\cite{linke} initiated the study of the rational (and real) counting function
from the Ehrhart viewpoint. She proved several
fundamental results starting with the fact that
$\rationalehr(\Pol;\lambda)$ 
is a 
\Def{quasipolynomial} in the rational (equivalently, real) variable $\lambda$, that
is,
\[
  \rationalehr\mleft(\Pol;\lambda\mright) \ = \ c_d\mleft(\lambda\mright) \, \lambda^d + c_{d-1}\mleft(\lambda\mright) \, \lambda^{d-1}
+ \dots + c_0\mleft(\lambda\mright) 
\]
where $c_0, c_1, \dots, c_d\colon \QQ\to\QQ$ are
periodic functions.
The least common period of $c_0(\lambda), \dots, c_d(\lambda)$ is the
\Def{period} of $\rationalehr\mleft(\Pol;\lambda\mright)$.
For $x \in \RR$, let $\lfloor x \rfloor$ (resp.\ $\lceil x \rceil$) denote the largest
integer $\leq x$ (resp.\ the smallest integer $\geq x$),
and $\{ x\} \coloneqq x - \lfloor x \rfloor$. 
Here is a first example, which we will revisit below:
  \begin{align}
   \rationalehr\mleft([1 ,2];\lambda\mright)
   \ &=\  \lfloor2\lambda\rfloor - \lceil\lambda\rceil +1 \\
   \ &=\  \begin{cases}
          n+1 & \text{ if } \lambda=n                   \quad \hfill \text{ for some }n\in \ZZ_{>0} \, , \\
          n & \text{ if } n<\lambda<n+\frac 1 2         \quad \hfill \text{ for some }n\in \ZZ_{>0} \, , \\
          n+1 & \text{ if } n+\frac 1 2 \leq\lambda<n+ 1\quad \hfill \text{ for some }n\in \ZZ_{>0} \, . 
         \end{cases}
  \end{align}
Rearranging gives the quasipolynomial in the format of the definition:
$$
\rationalehr\mleft([1 ,2];\lambda\mright) \ = \ \vol( [1,2])\lambda + c_0
(\lambda) \ = \ \lambda + (\{\lambda \} - \{2 \lambda \} ) \, .$$

Linke views the coefficient functions as piecewise-defined polynomials, which
allows her, among many other things, to establish differential equations
relating the coefficient functions.
Essentially concurrently,
Baldoni--Berline--K\"oppe--Vergne~\cite{baldoniberlinekoeppevergnerealdilation}, inspired
by~\cite{barvinokfixfirstcoeffs},
developed an algorithmic theory of \Def{intermediate sums} for polyhedra,
which includes $\rationalehr(\Pol;\lambda)$ as a special case.
We also mention more recent work of Royer~\cite{royerreconstruction,royersemireflexive},
which, among many other things, also studies rational Gorenstein polytopes (see below).

Our goal is to add a generating-function viewpoint
to~\cite{baldoniberlinekoeppevergnerealdilation,linke}, one that is inspired by \cite{stapledonweightedehrart,stapledonfreesums}. 
To set it up, we need
to make a definition.
Suppose the rational $d$-polytope $\Pol\subset\RR^d$ is 
given by the irredundant halfspace description
\begin{equation}\label{eq:irredundantPol}
  \Pol \ = \ \left\{\bx\in\RR^d : \, \bA\,\bx\leq \bb \right\},
\end{equation}
where $\bA\in\ZZ^{n\times d}$ and $\bb\in\ZZ^n$ such that the greatest common divisor of
$b_i$ and the entries in the $i$th  row of $\bA$ equals $1$, for every
$i \in \{1,\dots,n \}$.\footnote{
If $\Pol$ is a \emph{lattice} polytope then we do not need to include $b_i$ in this gcd
condition.
}
We define the \Def{codenominator} $r$ of $\Pol$ to be the least common multiple
of the nonzero entries of $\bb$: $$r \ \coloneqq \ \lcm\mleft(\bb\mright) \, .$$ 
As we assume that $\Pol$ is full dimensional, the codenominator is well-defined.
Our nomenclature arises from determining $r$ using duality, as follows.
Let $\Pol^\circ$ denote the relative interior of $\Pol$, and~let~$(\RR^d)^\vee$ be the dual vector space.
If $\Pol\subseteq \RR^d$ is a rational polytope such that $\mathbf 0 \in \Pol^\circ$, the
\Def{polar dual polytope} is $\Pol^\vee \coloneqq \{\bx \in (\RR^d)^\vee : \, \langle
\bx,\by\rangle \geq -1 \text{ for all } \by \in \Pol \}$, and $r = \min\{q \in \ZZ_{>0} :
\, q \, \Pol^{\vee} \text{ is a lattice polytope} \}$; see, e.g.,~\cite{barvinokbook}.

We will see in Section~\ref{sec:setup} that $\rationalehr(\Pol;\lambda)$ is fully determined by
evaluations at rational numbers with denominator $2r$ (see Corollary~\ref{cor:realehr} below
for details); if $\mathbf 0 \in \Pol$ then we actually need to know only evaluations at
rational numbers with denominator $r$.
Thus we associate two generating series to the rational Ehrhart counting function,  
 the \Def{rational Ehrhart series}, to a full-dimensional rational polytope $\Pol$ with codenominator~$r$:
 \begin{equation}\label{eq:defzerorationalEhrhart}
\rationalEhr{\Pol}{t} \ \coloneqq \ 1 + \sum_{n\in\ZZ_{> 0}}
\rationalehr\mleft(\Pol;\frac {n} {r}\mright)t^{\frac {n} {r}}
\end{equation}
and the \Def{refined rational Ehrhart series}
\[
	\refinedrationalEhr{\Pol}{t} \ \coloneqq \ 1 + \sum_{n\in\ZZ_{> 0}}
\rationalehr\mleft(\Pol;\frac {n} {2r}\mright)t^{\frac {n} {2r}} \,. 
\]
Continuing our comment above, we typically study $\rationalEhr{\Pol}{t}$ for polytopes such that $\mathbf 0 \in \Pol$, and $\refinedrationalEhr{\Pol}{t}$ for polytopes such that $\mathbf 0 \notin \Pol$.
Our first main result is as follows.

\begin{reptheorem}{thm:codem}

Let $\Pol \subseteq \RR^d$ be a rational $d$-polytope with codenominator $r$, and let $m \in \ZZ_{ >0 }$
such that~$\tfrac m r \Pol$ is a lattice polytope. Then
\[
  \rationalEhr{\Pol}{t}
  \ = \ \frac{\rationalhstar{\Pol}{t}{m}}{ \left( 1-t^{ \frac m r } \right)^{d+1}}
\]
where $\rationalhstar{\Pol}{t}{m}$ is a polynomial in $\ZZ [ t^\frac{1}{r}]$ with nonnegative integral coefficients. 
Consequently, $\rationalehr(\Pol; \lambda)$ 
is a quasipolynomial
and the period of 
$\rationalehr(\Pol; \lambda)$ divides~$\frac m r$, i.e., this period is of the form $\frac j r$ with $j\mid m$.
\end{reptheorem}

From this we recover Linke's result \cite[Corollary 1.4]{linke} that
$\rationalehr(\Pol;\lambda)$ is a quasi\-poly\-nomial with period dividing $q$, where $q$ is the
smallest positive rational number such that $q \Pol$ is a lattice polytope.

Section~\ref{sec:setup} contains structural theorems
about these generating functions: rationality and its consequences for the quasipolynomial
$\rationalehr(\Pol;\lambda)$ (\Cref{thm:codem} and \Cref{thm:codemx2}),
nonnegativity (\Cref{lem:rehr_quasipoly}), 
connections to the $\hstar$-polynomial in classical Ehrhart theory (\Cref{cor:hstarfromrhstar}),
and combinatorial reciprocity theorems (\Cref{coro:reciprocitycounting} and \Cref{cor:reciprocityseries}).

One can find a precursor of sorts to our generating functions $\rationalEhr{\Pol}{t}$ and
$\refinedrationalEhr{\Pol}{t}$ in work by Stapledon~\cite{stapledonweightedehrart,stapledonfreesums}, and in fact this work
was our initial motivation to look for and study rational Ehrhart generating functions. 
We explain the connection of~\cite{stapledonfreesums} to our work in
Section~\ref{sec:stapledon}. In particular, we deduce that in the case $\mathbf 0 \in
\Pol^\circ$ the generating function $\rationalEhr{\Pol}{t}$ exhibits additional symmetry (Corollary~\ref{cor:rhpalindromic}).

A $(d+1)$-dimensional, pointed, rational cone $\Con \subseteq \RR^{d+1}$ is called \Def{Gorenstein} 
if there exists a point $(p_0,\bp) \in \Con \cap \ZZ^{ d+1 }$ such that
$\Con^\circ \cap \ZZ^{ d+1 } = (p_0,\bp)+ \Con \cap \ZZ^{ d+1 }$ 
(see, e.g., \cite{batryrevborisov,brunsroemer,stanleyhilbertgradedalgebras}).
The point $(p_0,\bp)$ is called the \Def{Gorenstein point} of the cone.
We define the \Def{homogenization} $\hom(\Pol)\subset\RR^{d+1}$ of a rational polytope $\Pol= \{ \bx \in \RR^d : \, \bA \, \bx \le \bb \}$ as 
\begin{equation}
\label{eq:homogenization}
 \hom\mleft(\Pol\mright) \ \coloneqq \ \cone\mleft(\{1\}\times\Pol\mright) 
	\ \coloneqq \ \left\{ \left(x_0,\bx\right)\in\RR^{d+1}\, \colon\, \bA\bx\leq x_0\bb\,, \ x_0 \ge 0  \right\} .
\end{equation}
For a cone $C\subseteq\RR^{d+1}$, the \Def{dual cone} $C^{\vee} \subseteq (\RR^{d+1})^{\vee}$ is 
\begin{equation}
 C^{\vee} \ \coloneqq \ \left\{(y_0,\by) \in (\RR^{d+1})^{\vee}\, \colon \,
\langle(y_0, \by),(x_0,\bx)\rangle\geq 0 \text{ for all } (x_0,\bx) \in C \right\} .
\end{equation}
A lattice polytope $\Pol \subset \RR^d$ is \Def{Gorenstein} if 
the homogenization $\hom(\Pol)$ of $\Pol$
is Gorenstein; in the special case where the Gorenstein point of that cone
is $(1, \mathbf q)$, for some $\mathbf q \in \ZZ^d$, we call $\Pol$ \Def{reflexive} \cite{batyrevdualpoly,hibidual}.
Reflexive polytopes can alternatively be characterized as those lattice polytopes (containing
the origin) whose polar duals are also lattice polytopes, i.e., they have codenominator~1. This definition has a natural
extension to rational polytopes~\cite{fisetkasprzyk}. Gorenstein and reflexive polytopes
(and their rational versions) play an important role in Ehrhart theory, as they have palindromic $\hstar$-polynomials.
In Section~\ref{sec:gorenstein} 
we give the analogous result in rational Ehrhart theory \emph{without} reference to the
polar dual:

\begin{reptheorem}{thm:gorensteinv2}
Let $\Pol = \{ \bx \in \RR^d : \, \bA \, \bx \le \bb\} $ be a rational $d$-polytope with codenominator $r$ and $\mathbf 0\in \Pol$, as in Equation~\eqref{eq:irredundantPol} and Equation~\eqref{eq:splitPdescription}
Then the following are equivalent for $g,m \in \ZZ_{\geq 1}$ and $\frac{m}{r}\Pol$ a lattice polytope:
\begin{enumerate}[label=\textnormal{(\roman*)}]
\item \label{thm:gorenstein:gorenstein(intro)}
 $\Pol$ is $r$-rational Gorenstein with Gorenstein point $(g, \by) \in \hom (\frac{1}{r}\Pol)$.
\item \label{thm:gorenstein:eq(intro)}
There exists a (necessarily unique) integer solution $(g, \by)$ to
\begin{equation}\label{eq:gorenstein(intro)}
\begin{split}
 - \langle \ba_j ,\by \rangle &=1  \ \quad\text{ for }j=1,\dots,i\\
  b_j\,g-r\, \langle \ba_j,\by \rangle  &= b_j \quad \text{ for } j=i+1,\dots,n \,.
\end{split}
\end{equation}
\item \label{thm:gorenstein:palin(intro)}
$\rationalhstar{\Pol}{t}{m}$ is palindromic:
\[
  t^{(d+1)\frac{m}{r}-\frac g r} \, \rationalhstar{\Pol}{\frac 1 t}{m} \ = \ \rationalhstar{\Pol}{t}{m} \, .
\]
\item \label{thm:gorenstein:REhr(intro)}
$( -1)^{d+1}t^{\frac g r} \rationalEhr{\Pol}{t} = \rationalEhr{\Pol}{\frac 1 t}$.
\item \label{thm:gorenstein:rehr(intro)} $\rationalehr(\Pol;\frac n r)=\rationalehr(\Pol^\circ;\frac{n+g}{r})$ for all $n\in\ZZ_{\geq0}$.
\item \label{thm:gorenstein:dual(intro)}
$\hom ( \frac{1}{r} \Pol) ^\vee $ is the cone over a lattice polytope, i.e., there exists a lattice point $(g,\by)\in \hom(\frac 1 r \Pol)^\circ\cap\ZZ^{d+1}$ such that for every primitive ray generator $(v_0,\bv)$  
of $\hom(\frac 1 r \Pol)^\vee$
\begin{equation}
 \left\langle \left(g,\by\right), \left( v_0,\bv\right)
 \right\rangle \ = \ 1\,.
\end{equation}
\end{enumerate}

\end{reptheorem}
The equivalence of \ref{thm:gorenstein:gorenstein(intro)} and
\ref{thm:gorenstein:dual(intro)} is well known (see, e.g., \cite[Definition
1.8]{BatyrevNill} or \cite[Exercises~2.13 and~2.14]{brunsgubeladzektheory}).
We will see that there are many more \emph{rational} Gorenstein polytopes 
than among lattice polytopes; e.g., any rational polytope containing the origin in its interior
is rational Gorenstein (\Cref{cor:gorenstein}). 

We mention the recent notion of an $l$-reflexive polytope $\Pol$ (``reflexive of higher
index'') \cite{kasprzyknill}.
A lattice point $\bx \in \ZZ^d$ is \Def{primitive} if the $\gcd$ of its coordinates is equal to one. 
The \Def{$l$-reflexive polytopes} are precisely the lattice polytopes of the form
Equation~\eqref{eq:irredundantPol} with $\bb = (l, l, \dots, l)$ and primitive vertices; note that
this means $\Pol$ has codenominator $l$ and $\frac 1 l \Pol$ has denominator~$l$.

We conclude with two short sections further connecting our work to the existing literature.
Section~\ref{sec:symmdecomp} exhibits how one can deduce a theorem of
Betke--McMullen~\cite{betkemcmullen} (and also its rational
analogue~\cite{beckbraunvindasmelendez}) from rational Ehrhart theory.

Ehrhart's theorem gives an upper bound for the period of the quasipolynomial
$\ehr(\Pol; n)$, namely, the denominator of $\Pol$. When the period of
$\ehr(\Pol; n)$ is smaller than the denominator of $\Pol$, we speak of
\Def{period collapse}. One can witness this phenomenon most easily in the
Ehrhart series, as period collapse means that the rational function
in Equation~\eqref{eq:ehrgenfct} factors in such a way that one realizes there are no
nontrivial roots of unity that are poles.
It is an interesting question whether/how much period collapse happens in
rational Ehrhart theory, and how it compares to the classical scenario.
In Section~\ref{sec:periodcollapse}, we offer some data points for period collapse for
both rational and classical Ehrhart quasipolynomials.

\section{Rational Ehrhart Dilations}\label{sec:setup}

We assume throughout this article that all polytopes are full dimensional, 
and call a $d$-dimensional polytope in $\RR^d$ a \Def{$d$-polytope}.
We note that, consequently, the leading coefficient of $\ehr(\Pol; n)$ is constant (namely,
the volume of $\Pol$), and thus the rational generating function $\Ehr(\Pol;t)$ has a unique
pole of order $d+1$ at $t = 1$.
So we could write the rational generating function $\Ehr(\Pol;t)$ with denominator
$(1-t)(1-t^k)^d$; in other words,
$\hstar(\Pol;t)$ always has a factor $(1 +t + \dots +t^{k-1})$.
Recall, for $x \in \RR$, let $\lfloor x \rfloor$ (resp.\ $\lceil x \rceil$) denote the largest
integer $\leq x$ (resp.\ the smallest integer $\geq x$),
and $\{ x\} = x - \lfloor x \rfloor$.

 \begin{example}\label{example:running}
We feature the following line segments as running examples.
First, we compute the rational Ehrhart counting function. 
\begin{enumerate}[(i)]
\item $\Pol_1 \coloneqq \left[-1, \frac 2 3\right]$, codenominator $r=2$,
  \begin{align}
    \rationalehr\mleft(\Pol_1;\lambda\mright)
    \ &= \ \lceil \lambda \rceil +\left \lceil \tfrac 2 3 \lambda \right \rceil + 1 \\
      &= \ \begin{cases}
	    \frac 5 3 n + 1 \quad \text{ if } n \leq \lambda < n + \frac 1 2 \quad
\hfill \text{ for some } n  \in 3\ZZ_{>0} \, , \\
	    \frac 5 3 n + 1 \quad \text{ if } n+ \frac 1 2 \leq \lambda <  n + 1\quad \hfill \text{ for some } n \in 3\ZZ_{>0} \, , \\
	    \frac 5 3 n + 2 \quad \text{ if } n + 1\leq \lambda < n + \frac 3 2 \quad \hfill \text{ for some } n \in 3\ZZ_{>0} \, , \\
	    \frac 5 3 n + 3 \quad \text{ if } n + \frac 3 2\leq \lambda <n +  2 \quad \hfill \text{ for some } n \in 3\ZZ_{>0} \, , \\
	    \frac 5 3 n + 4 \quad \text{ if } n + 2 \leq \lambda <n +  \frac 5 2\quad \hfill \text{ for some } n \in 3\ZZ_{>0} \, , \\
	    \frac 5 3 n + 4 \quad \text{ if } n + \frac 5 2\leq \lambda < n+ 3  \quad \hfill \text{ for some } n \in 3\ZZ_{>0} \, . 
    \end{cases}
  \end{align}

\item $\Pol_2 \coloneqq \left[0, \tfrac 2 3\right]$, codenominator $r=2$,
  \begin{align}
	  \rationalehr\mleft(\Pol_2;\lambda\mright)
	   &\ =\ \left\lfloor \tfrac {2}{3}\lambda\right \rfloor +1 \\
	   &\ =\ 
	   \tfrac{2}{3}n+1  \quad \text{ if } n \leq \lambda < n + \tfrac 3 2 \quad
\hfill\text{ for some } n \in\tfrac 3 2 \ZZ_{> 0} \, .
  \end{align}
\item $\Pol_3 \coloneqq [1, 2]$, codenominator $r=2$,
  \begin{align}
   \rationalehr\mleft(\Pol_3;\lambda\mright)
   \ &=\  \lfloor2\lambda\rfloor - \lceil\lambda\rceil +1 \\
   \ &=\  \begin{cases}
          n+1 &\quad \text{ if } \lambda=n                   \quad \hfill\text{ for some }n\in \ZZ_{>0} \, , \\
          n &  \quad \text{ if } n<\lambda<n+\frac 1 2       \quad \hfill\text{ for some }n\in \ZZ_{>0} \, , \\
          n+1 &\quad \text{ if } n+\frac 1 2 \leq\lambda<n+ 1\quad \hfill\text{ for some }n\in \ZZ_{>0} \, . \\
         \end{cases}
  \end{align}
\vspace{1ex}

\item $\Pol_4 \coloneqq 2\Pol_3 = [2,4]$, codenominator $r=4$,
\begin{align}
   \rationalehr\mleft(\Pol_4;\lambda\mright)
   \ &=\ \lfloor 4 \lambda\rfloor -\lceil2\lambda\rceil + 1
    \ =\  \lfloor4\lambda\rfloor + \lfloor -2\lambda\rfloor +1\\
   \ &=\  2\lambda + 1-\{4\lambda\}+\{-2\lambda\} \\
   \ &=\  \begin{cases}
          2n+1 & \text{ if } \lambda=n                           \quad \hfill\text{ for
some }n\in \frac 1 2 \ZZ_{>0} \, ,\\
          2n & \text{ if } n<\lambda<n+\frac 1 4                 \quad \hfill\text{ for some }n\in \frac 1 2 \ZZ_{>0} \, ,\\
          2n+1 & \text{ if } n+\frac 1 4 \leq\lambda<n+ \frac 1 2\quad \hfill \text{ for some }n\in \frac 1 2 \ZZ_{>0} \, .
         \end{cases}
  \end{align}
\end{enumerate}
\end{example}

\begin{remark}\label{rem:codenom}
 If $\Pol$ is a lattice polytope, then the denominator of $\tfrac{1}{r} \Pol$ divides $r$.
On the other hand, the denominator of $\tfrac{1}{r} \Pol$ need not equal $r$, as can
be seen in the case of $\Pol_4$ above. 
\end{remark}
\begin{remark}
If $\tfrac{1}{r}\Pol$ is a lattice polytope, its Ehrhart polynomial is invariant under lattice translations. 
Unfortunately, this does not clearly translate to invariance of
$\rationalehr(\Pol;\lambda)$, as Linke already noted. 
Consider the line segment $[-1,1]$ and its translation $\Pol_4 = [2,4]$.
For any $\lambda \in (0,\frac {1}{4})$, we have $\rationalehr([-1,1];\lambda) = 1$ and $\rationalehr(\Pol_4 ; \lambda) = 0$.
This observation raises the following two related questions.
First, is there an example of a polytope and a translate with the same
codenominator? We expect the answer is ``no'' in dimension one.
Second, given a rational polytope $\Pol$, for which $r$ and $\tilde{\Pol}$ could $\Pol$ = $\tfrac{1}{r} \tilde{\Pol}$? 
Royer shows in \cite{royerreconstruction} that for every rational polytope $\Pol$ there
is a integral translation vector $\bv$ such that the functions $\rationalehr(k\bv+\Pol;
\lambda) $ are all distinct for $k \in \ZZ_{ \ge 0 }$. Moreover, polytopes can be uniquely identified by knowing the rational Ehrhart counting function for each integral translate of the polytope.
\end{remark}

\begin{lemma}\label{lem:monotone}
Let $\Pol \subseteq \RR^d$ be a rational $d$-polytope.
If $\mathbf 0 \in \Pol$, then $\rationalehr(\lambda)$ is monotone for $\lambda\in\QQ_{\geq0}$.
\end{lemma}

\begin{proof}
Let $\lambda < \omega$ be positive rationals. 
Suppose $\bx \in \RR^d$ and $\bx \in \lambda \Pol$. 
Then $\bx$ satisfies all $n$ facet-defining inequalities of $\lambda \Pol$: $\langle \ba_i, \bx\rangle \leq \lambda b_i\; \text{ for all } i \in [n]$.
If $b_i = 0$, then $\langle \ba_i,  \bx\rangle  \leq \lambda\cdot 0\ =\ \omega\cdot 0$.
Otherwise, $b_i >0$, and $\langle \ba_i,  \bx\rangle \leq \lambda b_i < \omega b_i$. 
So $\bx \in \omega \Pol$.
\end{proof}

\begin{proposition}\label{prop:constantintervals}
Let $\Pol \subseteq \RR^d$ be a rational $d$-polytope with codenominator~$r$. 
\begin{enumerate}[label=\textnormal{(\roman*)}]
\item \label{it:0notinp}The number of lattice points in $\lambda \Pol$ is constant for $\lambda \in (\tfrac{n}{r},\tfrac{n+1}{r}), \ n \in \ZZ_{\geq 0}$. 
\item \label{it:zeroinp} If $\mathbf 0 \in \Pol$, then the number of lattice points in $\lambda\Pol$ is constant for
$\lambda \in [\tfrac{n}{r},\tfrac{n+1}{r})$, $\ n \in \ZZ_{\geq 0}$.
\end{enumerate}
\end{proposition}

\begin{proof}
\ref{it:0notinp}. 
Suppose there exist two rationals $\lambda$ and  $\omega$ such that
$ \frac{n}{r} < \lambda < \omega < \frac{n+1}{r}$, and $\rationalehr(\lambda) \neq \rationalehr(\omega)$.
Then there exists $\bx \in \ZZ^d$ such that either $( \bx \in \omega \Pol \text{ and } \bx \notin \lambda \Pol )$ or $( \bx \in \lambda \Pol \text{ and } {\bx \notin \omega \Pol} )$.
Suppose  $( \bx \in \omega \Pol \text{ and } \bx \notin \lambda \Pol )$.
Then there exists a facet $F$ with integral, reduced inequality $\langle \ba,  \bv \rangle \leq b$ of $\Pol$ such that
\begin{align}
\langle \ba , \bx \rangle \leq \omega b, \;\;\;
\langle \ba , \bx \rangle > \lambda b,\text{ and} \;\;\;
\langle \ba , \bx \rangle \in \ZZ \, .
\end{align}
As $\lambda < \omega$, this implies $b > 0$. 
We have
$$
b \frac{n}{r} < \lambda b < \langle \ba ,\bx \rangle \leq \omega b < \frac{n+1}{r} b.
$$
As $r = b k$, with $k \in \ZZ_{>0}$, this is equivalent to 
\begin{equation}
  \label{eq:contradiction}
n < \lambda r < k \langle \ba,  \bx \rangle  \leq \omega r < n+1.
\end{equation}
This is a contradiction because $k \langle \ba,\bx\rangle $ is an integer.
The second case is proved analogously:
Assume~$\left( \bx \notin \omega \Pol \text{ and } \bx \in \lambda \Pol \right)$.
Then there exists again a facet $F$ with integral, reduced inequality $\langle \ba,  \bv \rangle \leq b$ of $\Pol$ such that
\begin{align}
\langle \ba , \bx \rangle > \omega b, \;\;\;
\langle \ba , \bx \rangle \leq \lambda b,\text{ and} \;\;\;
\langle \ba , \bx \rangle \in \ZZ.
\end{align}
As $\lambda < \omega$, this implies $b < 0$. 
We have
$$
 \frac{n+1}{r}\lvert b\rvert>\omega \lvert b\rvert > - \langle \ba ,\bx \rangle \geq \lambda \lvert b\rvert >\frac n r \lvert b\rvert \,.
$$
As $\frac{r}{\lvert b \rvert}\in\ZZ_{>0}$, this is equivalent to 
\begin{equation}\label{eq:nocontradiction}
 n+1>\omega r > - \frac{r}{\lvert b \rvert} \langle \ba ,\bx \rangle \geq \lambda r > n \,.
\end{equation}
This leads to the same contradiction.

\ref{it:zeroinp} 
If $\bzero\in\Pol$ we know that $\bb\geq\bzero$. So in the proof above only the first case applies.
(This can also be seen as a consequence of \Cref{lem:monotone}.) Allowing $\tfrac n r\leq\lambda$ leads, with the same computations, to the following weakened version of Equation~\eqref{eq:contradiction}:
$$
n \leq \lambda r < k \langle \ba,  \bx \rangle  \leq \omega r < n+1\,,
$$
which is still strong enough for the contradiction.
Note that this is not the case in Equation~\eqref{eq:nocontradiction}.
\end{proof}

 We define the \Def{real Ehrhart counting function}
\begin{equation}
 \realehr\mleft(\Pol;\lambda\mright) \ \coloneqq \ \left\lvert\lambda \Pol\cap\ZZ^d\right\rvert,
\end{equation}
for $\lambda\in\RR$.
It follows that we can compute the real Ehrhart function $\realehr$ from the rational Ehrhart function $\rationalehr$:

\begin{corollary}\label{cor:realehr}
Let $\Pol \subseteq \RR^d$ be a rational $d$-polytope with codenominator~$r$. Then
 \begin{equation}\label{eq:fromrationaltoreal}
  \realehr\mleft(\Pol;\lambda\mright) \ = \ \begin{cases}
                          \rationalehr\mleft(\Pol;\lambda\mright) &\text{ if }\lambda\in\frac 1 r \ZZ_{\geq0}\,,\\
                           \rationalehr\mleft(\Pol;\lfloor\lambda\rceil\mright) &\text{ if }\lambda\notin\frac 1 r \ZZ_{\geq0}\,,
                         \end{cases}
 \end{equation}
  where
  \begin{equation}
   \lfloor\lambda\rceil\coloneqq \frac{2j+1}{2r}\quad\text{for}\quad
   \left\lvert \lambda-\frac{2j+1}{2r}\right\rvert < \frac{ 1 }{ 2r } \quad\text{and}\quad j\in\ZZ \,. 
  \end{equation}
  In words, $ \lfloor\lambda\rceil$ is the element in $\frac{1}{2r}\ZZ$ with odd numerator that has the smallest Euclidean distance to $\lambda$ on the real line.
Furthermore, if $\mathbf 0 \in \Pol$ then
 \begin{equation}
  \realehr\mleft(\Pol;\lambda\mright) \ = \ \rationalehr \mleft( \Pol;\frac{\lfloor
r\lambda\rfloor}{r} \mright) . 
 \end{equation}
 \end{corollary}

In light of this Corollary, any statement  about the rational Ehrhart counting function
$\rationalehr(\lambda)$ in this paper generalizes to the real Ehrhart counting function $\realehr(\lambda)$ and we omit the latter versions for simplicity.
We proceed to prove one of the main results.
\begin{theorem}\label{thm:codem}
Let $\Pol \subseteq \RR^d$ be a rational $d$-polytope with codenominator $r$, and let $m \in \ZZ_{ >0 }$
such that~$\tfrac m r \Pol$ is a lattice polytope. Then
\[
  \rationalEhr{\Pol}{t}
  \ \coloneqq \sum_{n\in\ZZ_{\geq0}} \rationalehr\mleft(\Pol;\frac {n} {r}\mright)t^{\frac {n} {r}} 
  \ = \ \frac{\rationalhstarpolynomial(\Pol;t)}{ \left( 1-t^{ \frac m r } \right)^{d+1}}
\]
where $\rationalhstarpolynomial(\Pol;t)$ is a polynomial in $\ZZ [ t^\frac{1}{r}]$ with nonnegative integral coefficients. 
Consequently, the rational Ehrhart counting function $\rationalehr(\Pol; \lambda)$ is a 
quasipolynomial 
and the period of 
$\rationalehr(\Pol; \lambda)$ divides~$\frac m r$, i.e., this period is of the form $\frac j r$ with $j\mid m$.
\end{theorem}

\begin{proof} 
Our conditions imply that $\frac 1 r \Pol$ is a rational polytope with denominator
dividing $m$. Thus by standard Ehrhart theory,
\[
  \rationalEhr{\Pol}{t}
  \ = \ \Ehr \mleft( \tfrac 1 r \Pol; t^{ \frac 1 r } \mright)
  \ = \ \frac{\hstar \mleft( \tfrac 1 r \Pol; t^{ \frac 1 r } \mright) }{ \left( 1-t^{ \frac m r } \right)^{d+1}} \, ,
\]
and $\hstar ( \tfrac 1 r \Pol; t)$ has nonnegative integral coefficients.
\end{proof}

\begin{remark}
Our implicit definition of $\rationalhstarpolynomial(\Pol;t)$ depends on $m$.
We will sometimes use the notation $\rationalhstar{\Pol}{t}{m}$ to make this dependency
explicit.
Naturally, one often tries to choose $m$ minimal, which gives a canonical definition of
$\rationalhstarpolynomial(\Pol;t)$, but sometimes it pays to be flexible.
\end{remark}
\begin{remark}
Via Corollary~\ref{cor:realehr}, $\realehr(\Pol;\lambda)$ is a quasipolynomial
  and the period of 
$\realehr(\Pol; \lambda)$ divides~$\frac m r$, i.e., this period is of the form $\frac j r$ with $j\mid m$.
\end{remark}
\begin{remark}\label{rem:degreehstar}
By usual generatingfunctionology \cite{wilf}, the degree of $\rationalhstar{\Pol}{t}{m}$ is less than or equal to $m(d+1)-1$ as a polynomial in $t^{\frac 1 r}$.
\end{remark}

We also recover the following result of Linke~\cite{linke}.

\begin{corollary}
Let $\Pol \subseteq \RR^d$ be a rational $d$-polytope with codenominator $r$, and let $m \in \ZZ_{ >0 }$
such that $\frac m r \Pol$ is a lattice polytope. Then the period of the quasipolynomial
$\ehr(\Pol; \lambda)$ divides~$\frac m {\gcd(m,r)}$.
\end{corollary}

\begin{proof}
Viewed as a function of the integer parameter $n$, the function $\rationalehr(\Pol;\frac{n} {r})$ has period dividing~$m$. 
Thus $\ehr(\Pol; n) = \rationalehr(\Pol; n)$ has period
dividing~$\frac m {\gcd(m,r)}$.
\end{proof}

\begin{corollary}\label{lem:rehr_quasipoly}
Let $\Pol \subseteq \RR^d$ be a lattice $d$-polytope with codenominator $r$. Then
\[
  \rationalEhr{\Pol}{t}
  \ = \ \frac{\rationalhstar{\Pol}{t}{r}}{ \left( 1-t \right)^{d+1}}
\]
where $\rationalhstar{\Pol}{t}{r}$ is a polynomial in $\ZZ [ t^\frac{1}{r} ]$ with nonnegative coefficients. 
\end{corollary}

For polytopes that do not contain the origin, the following variant of Theorem~\ref{thm:codem} is useful.

\begin{theorem}\label{thm:codemx2}
Let $\Pol \subseteq \RR^d$ be a rational $d$-polytope with codenominator $r$, and let $m \in \ZZ_{ >0 }$
such that $\frac m {2r} \Pol$ is a lattice polytope. Then
\[
  \refinedrationalEhr{\Pol}{t} 
  \ \coloneqq \ 1 + \sum_{n\in\ZZ_{> 0}} \rationalehr\mleft(\Pol;\frac {n} {2r}\mright)t^{\frac {n} {2r}} 
  \ = \ \frac{\refinedrationalhstar{\Pol}{t}{m}}{ \left( 1-t^{ \frac m {2r} } \right)^{d+1}}
\]
where $\refinedrationalhstar{\Pol}{t}{m}$ is a polynomial in $\ZZ [ t^\frac{1}{2r} ]$ with nonnegative coefficients.
\end{theorem}

The proof of \Cref{thm:codemx2} is virtually identical to that of
Theorem~\ref{thm:codem}.
Similarly, many of the following assertions come in two versions, one for $\rationalEhr{\Pol}{t}$ and one for $\refinedrationalEhr{\Pol}{t}$.
We typically write an explicit proof for only one version, as the other is analogous. 

We recover another result of Linke~\cite{linke}.

\begin{corollary}
Let $\Pol \subseteq \RR^d$ be a lattice $d$-polytope.
The  rational Ehrhart function,
$\rationalehr(\Pol,\lambda)$, is given by a quasipolynomial of period~$1$.
\end{corollary}

\begin{corollary}\label{cor:hstarfromrhstar}
  If $\frac m r $ (resp.\ $\frac{m}{2r}$) in \Cref{thm:codem} (resp.\ \Cref{thm:codemx2}) is
integral we can retrieve the $\hstar$-polynomial from the $\rationalhstarpolynomial$-polynomial (resp.\ $\refinedrationalhstarpolynomial$-polynomial) by applying the operator $\Int$ that extracts from a polynomial in $\ZZ[t^{ \frac 1 r }]$ the
terms with integer powers of $t$: $\hstar(\Pol;t)=\Int(\rationalhstarpolynomial(\Pol;t))$ (resp.\ $\hstar(\Pol;t)=\Int(\refinedrationalhstarpolynomial(\Pol;t))$).
\end{corollary}

\begin{example}[continued]
	\label{example:running:series}
Here are the (refined) rational Ehrhart series of the running examples.
Recall that the  rational Ehrhart series of $\Pol$ in the variable $t$ can be computed as
the Ehrhart series of $\frac 1 r \Pol$ in the variable $t^{\frac 1 r}$ (resp.\ the refined rational Ehrhart as the Ehrhart series of $\frac{1}{2r} \Pol$ in the variable $t^{\frac{1}{2r}}$).
\begin{enumerate}[(i)]
\item $\Pol_1 \coloneqq [-1, \frac 2 3]$, $r$ = 2, $m =6$,
\begin{align}
\rationalEhr{\Pol_1}{t} \ &= \ \frac{1 + t^{\frac{1}{2}}+t+t^{\frac{3}{2}}+t^2}{\left(1-t\right)\left(1-t^{\frac{3}{2}}\right)}\\
				   \ &= \
\frac{1+t^{\frac{1}{2}}+2t+3t^{\frac{3}{2}}+4t^2+4t^{\frac{5}{2}}+4t^3+4t^{\frac{7}{2}}+3t^4+2t^{\frac{9}{2}}+t^5+t^{\frac{11}{2}}
}{\left(1-t^3\right)^2} \, .
\end{align}
\item $\Pol_2 \coloneqq [0, \frac 2 3]$, $r$ = 2, $m$ = 3,
	\begin{equation}
		\rationalEhr{\Pol_2}{t} \ = \ \frac{1}{\left(1-t^{\frac 1 2 }\right)\left(1-t^{\frac 3 2}\right)} = \frac{1+ t^{\frac 1 2}+ t}{\left(1-t^{\frac 3 2}\right)^2} \, .
	\end{equation}
\item $\Pol_3 \coloneqq [1, 2]$, $r=2$. $\frac 1 4 \Pol_3=[\tfrac 1 4 , \tfrac 1 2]$ and $m=4$, so $\frac{m}{2r}=1 $.
See \Cref{fig:ex12}.
 \begin{equation}
 \begin{split}
  \refinedrationalEhr{P_3}{t}\ &=\ \frac{1+t^{\frac 1 2}+t^{\frac 3 4}+t^{\frac 5 4}}{\left(1-t\right)^2}\ 
  =\ \frac{\left(1+t^{\frac 3 4}\right)\left(1+t^{\frac 1 2}\right)}{\left(1-t\right)^2} \, .
  \end{split}
 \end{equation}
 \begin{figure}
 \centering
  \includegraphics[width=.5\textwidth]{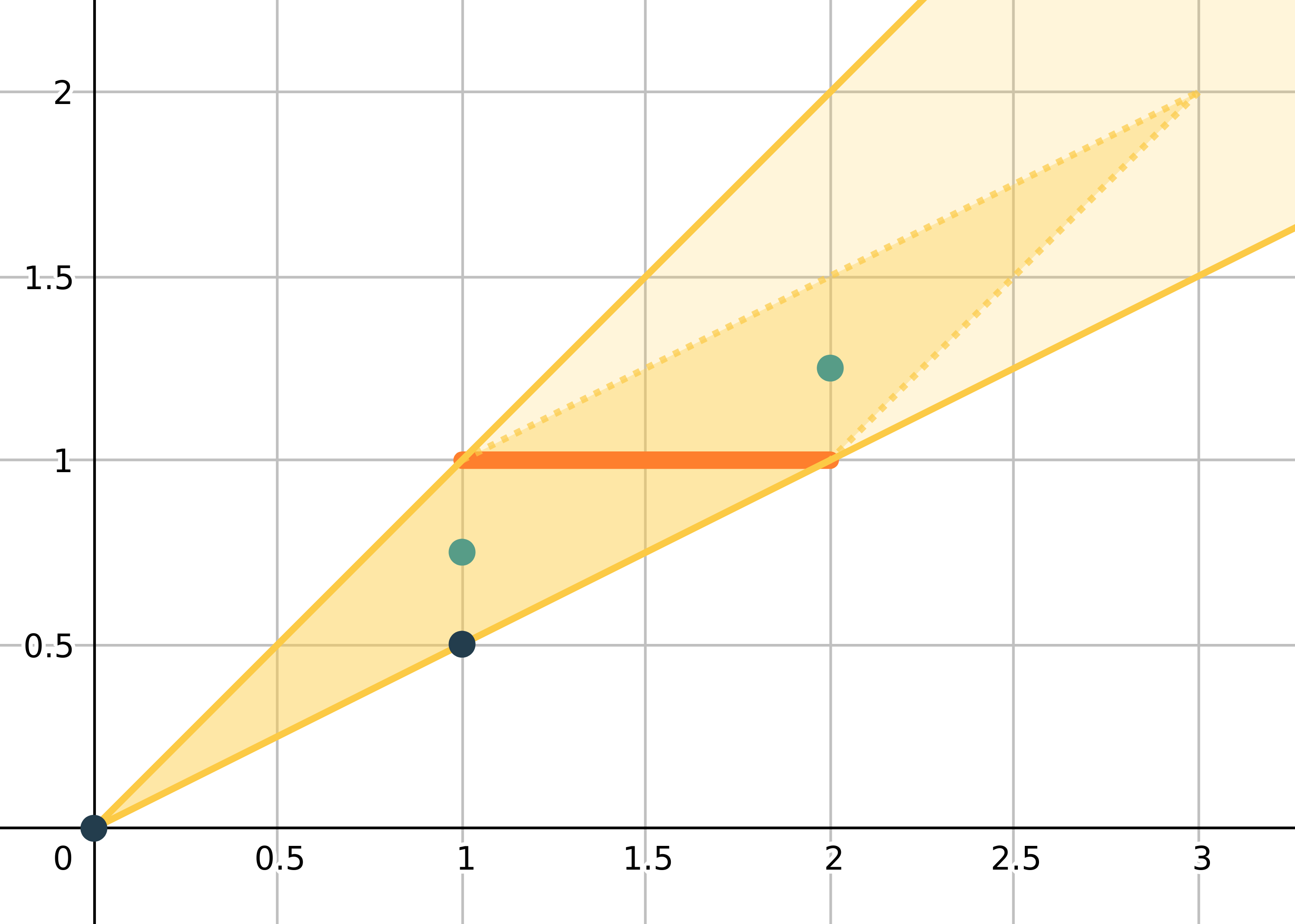}
  \caption{The cone $\hom(\Pol_3)$ over $\Pol_3=[1,2]$. The lattice points in the fundamental parallelepiped with respect to the lattice $\frac{1}{4}\ZZ\times\ZZ$ are $(0,0)$, $(\frac 1 2, 1)$, $(\frac 3 4, 1) $, $(\frac 5 4, 2)$. }\label{fig:ex12}
 \end{figure}

 \item $\Pol_4 \coloneqq [2,4]$, $r=4$. \label{ex:unimodal}
  Then $\frac 1 8 \Pol_4=[\tfrac 1 4 , \tfrac 1 2]$ and $m=4$, 
  so $\frac{m}{2r}=\frac 1 2$.
   See \Cref{fig:ex24}.
 \begin{align}
  \hspace{-1pt}\refinedrationalEhr{\Pol_4}{t}
  \ &= \ \frac{1+t^{\frac 1 4}+t^{\frac 3 8}+2t^{\frac 1 2}+t^{\frac 5 8}+2t^{\frac 3 4}+2t^{\frac 7 8}+t+2t^{\frac 9 8}+t^{\frac 5 4}+t^{\frac{11}{8}}+t^{\frac{13}{8}}}{\left(1-t\right)^2}\\
  &= \ \frac{1 + t^{\frac 1 4} +t^{\frac 3 8} + t^{\frac5 8}}{(1-t^{\frac 1 2})^2} \, .
 \end{align}
 Choosing $m$ to be minimal means $\refinedrationalhstar{\Pol_4}{t}{4}=(1+t^{\frac 3 8})(1+t^{\frac 1 4})=1+t^{\frac 1 4}+t^{\frac 3 8}+t^{\frac 5 8}=\refinedrationalhstar{\Pol_3}{t^{\frac 1 2}}{4}$.
 The rational Ehrhart
counting function agrees with a quasipolynomial for $\lambda\in\frac{1}{2r}\ZZ$.
\begin{figure}
\centering
\includegraphics[width=\textwidth]{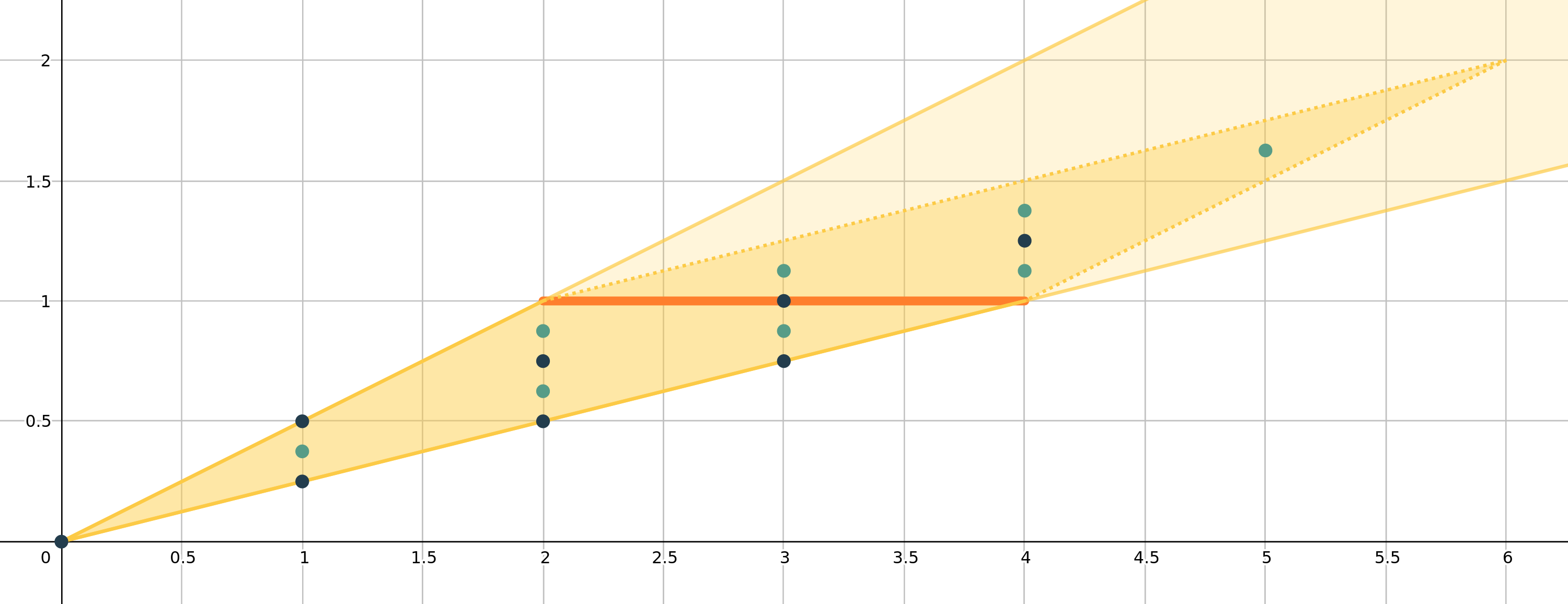}
\caption{The cone $\hom(\Pol_4)$ over $\Pol_4=[2,4]$. The lattice points in the fundamental parallelepiped with respect to the lattice $\frac{1}{8}\ZZ\times\ZZ$ are shown in the figure.}\label{fig:ex24}
\end{figure}
\end{enumerate}
From the (refined) rational Ehrhart series of these examples, we can recompute the quasipolynomials found earlier.
For example, for $\Pol_3$:

\begingroup
\allowdisplaybreaks
 \begin{equation}
 \begin{split}
  \refinedrationalEhr{\Pol_3}{t} \ &= \ \frac{1+t^{\frac 1 2}+t^{\frac 3 4}+t^{\frac 5 4}}{\left(1-t\right)^2} 
  = \ \left(1+t^{\frac 1 2}+t^{\frac 3 4}+t^{\frac 5 4}\right)\sum_{j\geq 0}\left(j+1\right)t^j\\
  &=  \ \sum_{j\geq 0}\left(j+1\right)t^j + \sum_{j\geq 0}\left(j+1\right)t^{j+\frac 1 2} 
  +\sum_{j\geq 0}\left(j+1\right)t^{j+\frac 3 4} + \sum_{j\geq 0}\left(j+1\right)t^{j+\frac 5
4} \, .
  \end{split}
 \end{equation}
 \endgroup
With a change of variables we compute for $ \lambda \in\frac 1 4 \ZZ$
\begin{equation}
 \rationalehr\mleft(\lambda\mright) \ = \ \begin{cases}
                        \lambda +1 & \text{ if } \lambda\in\ZZ, \\
                        \lambda -\frac 1 4 & \text{ if } \lambda \equiv \frac 1 4 \bmod 1, \\
                        \lambda +\frac 1 2 & \text{ if } \lambda \equiv \frac 1 2 \bmod 1, \\
                        \lambda +\frac 1 4 & \text{ if } \lambda \equiv \frac 3 4 \bmod 1. \\
                       \end{cases}
\end{equation}
\end{example}

Next we recover the reciprocity result for the rational Ehrhart function of rational polytopes proved by Linke \cite[Corollary 1.5]{linke}.
\begin{corollary}\label{coro:reciprocitycounting}
Let $\Pol \subseteq \RR^d$ be a rational $d$-polytope. Then $(-1)^d \, \rationalehr(\Pol;-\lambda)$ equals the
number of interior lattice points in $\lambda \Pol$, for any $\lambda > 0$.
\end{corollary}

\begin{proof}
Let $\Pol \subseteq \RR^d$ be a rational $d$-polytope with codenominator~$r$. 
The fact that  $\rationalehr(\Pol;\lambda)$ 
 is a quasipolynomial  allows us to
extend Equation~\eqref{eq:fromrationaltoreal} to the negative (and therefore all) rational numbers via
 \begin{equation}
  \rationalehr\mleft(\Pol;\lambda\mright) =
\rationalehr\mleft(\Pol;\lfloor\lambda\rceil\mright) \qquad \text{ if }\lambda\notin\tfrac 1 r \ZZ \,.
 \end{equation}
By standard Ehrhart--Macdonald Reciprocity, $(-1)^d \rationalehr(\Pol; - \frac n {2r}) = \ehr(\frac 1 {2r} \Pol; -n )$ equals the number of lattice points in the interior of $\frac n {2r}\Pol$. 
The result now follows from $\lfloor-\lambda\rceil = - \lfloor\lambda\rceil$.
\end{proof}

Let $\Pol \subseteq \RR^d$ be a rational $d$-polytope, let $\Pol^\circ$ denote its interior and $\rationalehr(\Pol^\circ;\lambda)\coloneqq \lvert\lambda\Pol^\circ\cap\ZZ^d\rvert$.
We define the (refined) rational Ehrhart series of the interior of a polytope as follows:
\begin{align}
 \rationalEhr{\Pol^\circ}{t} \ &\coloneqq \sum_{\lambda\in\frac 1 r \ZZ_{>0}}
\rationalehr\mleft(\Pol^\circ;\lambda\mright) \, t^\lambda \, ,\\
  \refinedrationalEhr{\Pol^\circ}{t} \ &\coloneqq \sum_{\lambda\in\frac{ 1}{2r} \ZZ_{>0}} \rationalehr\mleft(\Pol^\circ;\lambda\mright) \, t^\lambda\, ,
\end{align}
where $r$ as usual denotes the codenominator of $\Pol$.

\begin{corollary}\label{cor:reciprocityseries}
 Let $\Pol \subseteq \RR^d$ be a rational $d$-polytope with codenominator $r$, and let $m\in\ZZ_{>0}$ be such that $\frac{m}{r}\Pol$ is a lattice polytope.
 \begin{enumerate}[label=\textnormal{(\roman*)}]
 \item \label{hstarint} The rational Ehrhart series of the open polytope $\Pol^\circ$ has the rational expression
  \begin{align}
  \rationalEhr{\Pol^\circ}{t} \ = \ \frac{\rationalhstar{\Pol^\circ}{t}{m}}{\left(1-t^{\frac m r}\right)^{d+1}} 
 \end{align}
 where $\rationalhstar{\Pol^\circ}{t}{m}$  is a polynomial in $\ZZ[t^{\frac 1 r}]$.
 \item \label{ehrseriesint} The rational Ehrhart series fulfills the reciprocity relation 
 \begin{align}
  \rationalEhr{\Pol^\circ}{t} \ = \ \left(-1\right)^{d+1}\rationalEhr{\Pol}{\frac 1 t}\,.
 \end{align}
 \item \label{rathstarintandout} The $\rationalhstarpolynomial$-polynomial of the polytope $\Pol$ and its interior $\Pol^\circ$ are related by
	\begin{align}
        \rationalhstar{\Pol^\circ}{t}{m} \ = \ \left({t}^{\frac m r}\right)^{d+1}\rationalhstar{\Pol}{\frac{1}{t}}{m}\,.
	\end{align}
 \end{enumerate}
\end{corollary}
\begin{proof}
Identity \ref{hstarint} follows from Ehrhart--Macdonald reciprocity (see, e.g., \cite[Theorem~$4.4$]{ccd}) and \Cref{rem:degreehstar}:
 \begin{align}
  \rationalEhr{\Pol^\circ}{t} 
  \ &= \sum_{\lambda\in\frac 1 r \ZZ_{>0}} \rationalehr\mleft(\Pol^\circ;\lambda\mright) t^\lambda
  \ = \sum_{n\in \ZZ_{>0}} \ehr\mleft(\frac 1 r\Pol^\circ;n\mright) t^{\frac n r}
  \ = \ \Ehr\mleft(\frac 1 r \Pol^\circ;t^{\frac 1 r }\mright) \\
  \ &= \ \left(-1\right)^{d+1} \Ehr\mleft(\frac 1 r \Pol;t^{-\frac 1 r }\mright) 
  \ = \ \left(-1\right)^{d+1}\frac{\hstar\mleft(\frac{1}{r }\Pol;t^{-\frac 1 r} \mright)}{\left(1-t^{-\frac m r} \right)^{d+1}} 
  \ = \ \frac{\left(t^{\frac m r}\right)^{d+1} \hstar\mleft(\frac{1}{r
}\Pol;t^{-\frac 1 r} \mright)}{\left(1-t^{\frac m r} \right)^{d+1}} \, .
\end{align}
For identities \ref{ehrseriesint} and \ref{rathstarintandout} we again apply
Ehrhart--Macdonald reciprocity:
\begingroup
\allowdisplaybreaks
\begin{align}
\frac{\left({t}^{\frac m r}\right)^{d+1}\rationalhstar{\Pol}{\frac{1}{t}}{m}}{\left(1-{t}^{\frac m r}\right)^{d+1}} 
\ &= \ \frac{\left( -1\right)^{d+1}\rationalhstar{\Pol}{\frac{1}{t}}{m}}{\left(1-{\left(\frac{1}{t}\right)}^{\frac {m}{ r}}\right)^{d+1}}
\ = \ \left(-1\right)^{d+1} \rationalEhr{\Pol}{\frac 1 t} \\ 
&=  \ \left(-1\right)^{d+1} \Ehr \mleft( \frac 1 r \Pol ; \frac {1}{t^{\frac{1}{r}}} \mright) 
\ = \ \Ehr \mleft( \frac{1}{r} \Pol^\circ ; t^{\frac 1 r} \mright)\\
&= \sum_{\lambda \in \ZZ_{> 0}} \ehr \mleft( \frac{1}{r}\Pol^{\circ} ; \lambda   \mright)t^{\frac \lambda r} 
\ = \sum_{\lambda \in \frac{1}{r}\ZZ_{>0}} \rationalehr \mleft(\Pol^{\circ} ; \frac{\lambda }{r}  \mright)t^{\frac \lambda r} \\
&= \ \rationalEhr{\Pol^\circ}{t} 
\ = \ \frac{\rationalhstar{\Pol^\circ}{t}{m} }{\left(1-t^{\frac m
r}\right)^{d+1}}\,. \qedhere
\end{align}
\endgroup
\end{proof}

As usual there is a refined version:

\begin{corollary}\label{cor:reciprocityseriesrefined}
 Let $\Pol \subseteq \RR^d$ be a rational $d$-polytope with codenominator $r$, and let $m\in\ZZ_{>0}$ be such that $\frac{m}{2r}\Pol$ is a lattice polytope.
 \begin{enumerate}[label=\textnormal{(\roman*)}]
 \item 
 The refined rational Ehrhart series of the open polytope $\Pol^\circ$ have the rational expressions 
  \begin{align}
  \refinedrationalEhr{\Pol^\circ}{t} \ = \ \frac{\refinedrationalhstar{\Pol^\circ}{t}{m}}{\left(1-t^{\frac{ m}{2 r}}\right)^{d+1}}\,,
 \end{align}
 where  $ \refinedrationalhstar{\Pol^\circ}{t}{m}$ is a polynomial in  $\ZZ[t^{\frac{ 1}{2 r}}]$.
 \item 
 The refined rational Ehrhart series fulfills the reciprocity relation 
 \begin{align}
  \refinedrationalEhr{\Pol^\circ}{t} \ = \ \left(-1\right)^{d+1}\refinedrationalEhr{\Pol}{\frac 1 t}\,.
 \end{align}
 \item 
 The $\refinedrationalhstarpolynomial$-polynomial of the polytope $\Pol$ and its interior $\Pol^\circ$ are related by
	\begin{align}
        \refinedrationalhstar{\Pol^\circ}{t}{m} \ = \ \left({t}^{\frac{m}{2 r}}\right)^{d+1}\refinedrationalhstar{\Pol}{\frac 1 t}{m}\,.
	\end{align}
 \end{enumerate}
\end{corollary}

\begin{remark}
The \Def{codegree} of a lattice polytope is defined as $\dim(\Pol)+1 - \deg(h^*(t))$.
Analogously, in the rational case, we define the \Def{rational codegree of
$\rationalhstar{\Pol}{t}{m}$} to be $$\frac{m}{r}(\dim(\Pol)+1) -
\deg(\rationalhstar{\Pol}{t}{m}) \, ,$$ where the degree of $\rationalhstar{\Pol}{t}{m}$ is its (possibly fractional) degree as a polynomial in $t$.
Likewise, the \Def{rational codegree of $\refinedrationalhstar{\Pol}{t}{m}$} is defined as $\frac{m}{2r}(\dim(\Pol)+1)- \deg(\refinedrationalhstar{\Pol}{t}{m})$.
As in the integral case, the rational codegree of $\rationalhstar{\Pol}{t}{m}$ is the smallest integral dilate of $\frac 1 r \Pol$ containing interior lattice points. 
The proof requires no new insights and we omit it here.
\end{remark}

\section{Stapledon}\label{sec:stapledon}

We recall the setup from~\cite{stapledonfreesums}.
Let $\Pol\subseteq \RR^d$ be a lattice $d$-polytope with codenominator $r$ and $\mathbf 0\in\Pol$.
Let $\partial_{\neq 0}(\Pol)$ denote the union of facets of $\Pol$ that do not contain the origin. 
In order to study all rational dilates of the boundary of $\Pol$, Stapledon introduces the generating function
\begin{equation}\label{eq:stapledonsetup}
 \WEhr\mleft(\Pol;t\mright) \ \coloneqq \ 1 + \sum_{\lambda\in\QQ_{> 0 }} \left| \partial_{\neq
0}\mleft(\lambda\Pol\mright)\cap \ZZ^d \right| t^\lambda
 \ = \ \frac{ \htilde\mleft(\Pol;t\mright)}{\left(1-t\right)^d} \, ,
\end{equation}
where $\htilde(\Pol;t)$ is a polynomial in $\ZZ[t^{\frac 1 r}]$ with fractional exponents.
The generating function $\WEhr$ is closely related to the (rational) Ehrhart series: 
the truncated sum $1 + \sum_{\lambda \in \QQ >0 }^\omega | \partial_{\neq 0} (\lambda \Pol) \cap \ZZ^d  |$ equals the number of lattice points in$~\omega \Pol$. 
\Cref{prop:constantintervals} allows us to discretize this sum:
\begin{corollary}\label{coro:boundarypickup}
 Let $\Pol\subseteq \RR^d$ be a lattice $d$-polytope with codenominator $r$ and $\mathbf 0\in\Pol$. 
 The number of lattice points in $\lambda \Pol$ equals $1 + \sum_{\omega \in \frac{1}{r}\ZZ_{>0},\, \omega < \lambda} |\partial_{\neq 0} (\omega \Pol) \cap \ZZ^d  |$.
\end{corollary}
\begin{proof}
 As $\mathbf 0 \in \Pol$, every nonzero lattice point in $\lambda \Pol$ occurs in $ \partial_{\neq 0}(\omega\Pol)$ for some unique $\omega \in \QQ$ where
$0 < \omega \leq \lambda$. 
Using \Cref{lem:monotone},
$$ \lambda \Pol \cap \ZZ^d \ = \ \mathbf{0} \cup \bigsqcup_{\omega \in \QQ_{>0}}^{\lambda} ( \partial_{\neq_0}(\omega \Pol) \cap \ZZ^d )\,.  $$
By \Cref{prop:constantintervals}, the union $ \bigsqcup_{\omega \in \QQ_{>0}}^{\lambda} ( \partial_{\neq_0}(\omega \Pol) \cap \ZZ^d ) $ is discrete and disjoint. 
\end{proof}
Similarly, $\htilde(\Pol;t)$ is related to $\hstar(\tfrac{1}{r}\Pol;t^{\frac{ 1}{ r}}) $ and to $\rationalhstar{\Pol}{t}{m}$, as we show in \Cref{lem:hstarandtilde} and \Cref{rationalhstarandhtilde}. 
Recall that we use $\rationalhstar{\Pol}{t}{m}$ to keep track of the denominator of 
${\rationalEhr{\Pol}{t} = \tfrac{\rationalhstar{\Pol}{t}{m}}{( 1-t^{\frac m r}) ^{d+1}}}$.

\begin{lemma}\label{lem:hstarandtilde}
Let $\Pol\subseteq \RR^d$ be a lattice $d$-polytope with codenominator $r$ such that $\mathbf 0\in\Pol$.
Let $k$ be the denominator of $\frac{1}{r}\Pol$. Then
 \begin{equation}
  \hstar\mleft(\frac{1}{r}\Pol;t^{\frac 1 r}\mright)
  \ = \ \frac{\left(1-t^{\frac k r}\right)^{d+1}}{\left(1-t^{\frac 1 r}\right)\left(1-t\right)^d}\, \htilde\mleft(\Pol;t\mright)\,.
 \end{equation}
\end{lemma}
\begin{proof}
Applying classical Ehrhart theory, \Cref{prop:constantintervals} and
\Cref{coro:boundarypickup}, we compute 
\begingroup
\allowdisplaybreaks
\begin{align}
\frac{\hstar\mleft(\frac{1}{r}\Pol;t^{\frac 1 r}\mright)}{\left(1-t^{\frac k r}\right)^{d+1}} 
\ =\ \Ehr\mleft(\frac{1}{r} \Pol;t^{\frac 1 r}\mright)
\ =& \ 1+ \sum_{n\in\ZZ_{>0}} \ehr\mleft(\frac 1 r \Pol;n\mright)t^{\frac n r} \\ 
 {=}& \ 1+ \sum_{n\in\ZZ_{>0}} \left(1+ \sum_{j=1}^n \left | \partial_{\neq 0}\mleft(\frac{j}{r}\Pol\mright)\cap \ZZ^d \right| \right)t^{\frac n r}\\ 
  {=}&\ 1+ \sum_{n\in\ZZ_{>0}} t^{\frac n r}  + \sum_{j> 0} \sum_{n\geq j} \left | \partial_{\neq 0}\mleft(\frac{j}{r}\Pol\mright)\cap \ZZ^d \right|t^{\frac n r} \\
 {=}&\ 1+ \frac{t^{\frac 1 r}}{1- t^{\frac 1 r}} + \sum_{j> 0}  \left | \partial_{\neq 0}\mleft(\frac{j}{r}\Pol\mright)\cap \ZZ^d \right |  \sum_{n\geq j} t^{\frac n r} \\
 {=}&\ \frac{ 1 - t^{\frac 1 r} + t^{\frac 1 r} + \sum_{j> 0}  \left| \partial_{\neq 0}\mleft(\frac{j}{r}\Pol\mright)\cap \ZZ^d \right|  t^{\frac j r}   }{ 1 - t^{\frac 1 r}} \\
  {=}&\ \frac{\WEhr\mleft(\Pol;t\mright)}{1-t^{\frac 1 r}}\
 {=}\
\frac{\htilde\mleft(\Pol;t\mright)}{\left(1-t^{\frac 1 r}\right)\left(1-t\right)^d} \, . \qedhere
\end{align} 
\endgroup
\end{proof}
\begin{remark}
The factor multiplying $\htilde(\Pol;t)$ in \Cref{lem:hstarandtilde} can be rewritten in terms of finite geometric series. 
Let the codenominator $r=ks$ for some $s\in\ZZ_{\geq 1}$ (by \Cref{rem:codenom}).
Rewriting yields 
  \begin{equation}
  \begin{split}
   \frac{\left(1-t^{\frac k r}\right)^{d+1}}{\left(1-t^{\frac 1
r}\right)\left(1-t\right)^d} \ &= \ \frac{\left(1-t^{\frac k r}\right)}{\left(1-t^{\frac 1 r}\right)} \left(\frac{\left(1-t^{\frac k r}\right)}{\left(1-t\right)}\right)^d \\
   &= \ \frac{\left(1-t^{\frac 1 s}\right)}{\left(1-t^{\frac 1 {ks}}\right)} \left(\frac{1}{1+t^{\frac 1 s} + \dots+t^{\frac{s-1}{s}}}\right)^d
   = \ \frac{1+t^{\frac 1 r}+\dots+t^{\frac{k-1}{r}}}{\left(1+t^{\frac 1 s} + \dots+t^{\frac{s-1}{s}}\right)^d}\,.
  \end{split}
\end{equation}
 If $k=r$, this simplifies to $(1+t^{\frac 1 r} + \dots + t^{\frac {r-1}{r} })$.
\end{remark}

\begin{remark}
 \Cref{lem:hstarandtilde} corrects \cite[Remark 3]{stapledonfreesums}, which was missing the factor between $\hstar(\frac{1}{ r}\Pol;t^{\frac 1 r})$ and $\htilde(\Pol;t)$.
\end{remark}

\begin{corollary}\label{rationalhstarandhtilde}
Let $\Pol\subseteq \RR^d$ be a lattice $d$-polytope with codenominator $r$ such that $\mathbf 0\in\Pol$.
Let $k$ be the denominator of $\frac{1}{r}\Pol$. 
Then
\begin{equation}
\rationalhstar{\Pol}{t}{k}
\ = \ \hstar \mleft( \tfrac{1}{r}\Pol;t^{\frac 1 r} \mright) 
\ = \ \frac{\left(1-t^{\frac k r}\right)^{d+1}}{\left(1-t^{\frac 1
r}\right)\left(1-t\right)^d} \, \htilde\mleft(\Pol,t\mright)\,.
\end{equation}
\end{corollary}

\begin{remark}
 In \cite[Equation (14)]{stapledonweightedehrart} and \cite[Equation (6)]{stapledonfreesums},  
Stapledon shows that  $\hstar(\Pol;t)=\Psi(\htilde(\Pol;t))$,
 where $\Psi\colon \bigcup_{r\in\ZZ_{>0}}\RR[t^{\frac 1 r}]\to\RR[t]$ is defined by $\Psi(t^\lambda)=t^{\lceil\lambda\rceil}$.
In the case of a lattice polytope with $\frac m r \in\ZZ$ we give a different construction to recover the $\hstar$-polynomial from the $\refinedrationalhstarpolynomial$- and $\rationalhstarpolynomial$-polynomial by applying the operator $\Int$ (see  \Cref{cor:hstarfromrhstar}).
 \Cref{rationalhstarandhtilde} shows that, after a bit of computation,  these two constructions are equivalent. 
\end{remark}
 
 \begin{remark}
 For a lattice $d$-polytope $\Pol\subseteq\RR^d$  with codenominator $r$, $\mathbf
0\in\Pol$, and denominator of $\frac{1}{2r}\Pol$ equal to $k$,  we can relate
$\refinedrationalhstar{\Pol}{t}{k}$ and $\hstar(\frac{1}{2r}\Pol;t^{\frac {1}{2r}})$ in a similar way.
 We again write $\refinedrationalhstar{\Pol}{t}{k}$ to emphasize that it is the numerator of $\frac{\refinedrationalhstar{\Pol}{t}{k}}{(1-t^{\frac{k}{2r}})^{d+1}}$.
Then
\begin{equation}
	\refinedrationalhstar{\Pol}{t}{k} \ = \ \hstar \mleft(\frac{1}{2r}\Pol;t^{\frac{1}{2r}} \mright)
	\ = \ \frac{\left(1-t^{\frac{k}{2r}}\right)^{d+1}}{\left(1-t^{\frac{1}{2r}}\right)\left(1-t\right)^d} \, \htilde\mleft(\Pol;t\mright)\, .
\end{equation}
 \end{remark}
 \begin{corollary}\label{cor:rhpalindromic}
 Let $\Pol\subseteq\RR^d$ be a lattice $d$-polytope with $\mathbf 0\in\Pol^\circ$. Let $r$ be the codenominator of $P$ and $k$ be the denominator of $\frac{1}{r}\Pol$. 
 Then $\rationalhstar{\Pol}{t}{k}$ is palindromic.
\end{corollary}

\begin{proof}
From \cite[Corollary~$2.12$]{stapledonweightedehrart} we know that $\htilde(\Pol;t)$ is palindromic if $ \mathbf 0\in\Pol^\circ$.
 We compute, using \Cref{rationalhstarandhtilde},
 
\begin{equation}
 \begin{split}
	 \rationalhstar{\Pol}{t^{-1}}{k} \ &= \ \frac{\left(1-t^{\frac {-k}{ r}}\right)^{d+1}}{\left(1-t^{\frac {-1}{ r}}\right)\left(1-t^{-1}\right)^d} \ \htilde\mleft(\Pol;t^{-1}\mright)\\
					&= \ \frac{t^{\frac{-\left(d+1\right)k}{r}}}{t^{\frac{-1}{r}}}\frac{\left(1-t^{\frac {k}{ r}}\right)^{d+1}}{\left(1-t^{\frac {1}{ r}}\right)\left(1-t\right)^d} \ \htilde\mleft(\Pol;t\mright)
					= \ \frac{1}{t^{\frac{k\left(d+1\right)-1}{r}}}\, \rationalhstar{\Pol}{t}{k} \, .
 \end{split}
\end{equation}
Note that this implies, since the constant term of $\rationalhstar{\Pol}{t}{k}$
is 1, that the degree of $\rationalhstar{\Pol}{t}{k}$ (measured as a polynomial in $t^{ \frac 1 r }$) equals $k(d+1)-1$.
\end{proof}

This suggests that there is a 3-step hierarchy for rational dilations: 
$\mathbf 0 \in \Pol^\circ$ comes with extra symmetry, $\mathbf 0 \in \Pol$ 
comes with \Cref{prop:constantintervals} \ref{it:zeroinp} and so we ``only''
have to compute $\rationalhstar{\Pol}{t}{k} \in \ZZ[t^{\frac 1 r}]$, and
$\mathbf 0 \notin \Pol$ means we have to compute
$\refinedrationalhstar{\Pol}{t}{k} \in \ZZ[t^{\frac 1 {2r}}]$. 
Corollary \ref{cor:rhpalindromic} is related to Gorenstein properties of rational polytopes,
which we consider in the next section.

\section{Gorenstein Musings}\label{sec:gorenstein}

Our main goal in this section is to extend the notion of Gorenstein polytopes to the rational case. 
A rational $d$-polytope $\Pol \subseteq \RR^d$  is \Def{$\gam$-rational Gorenstein} if $\hom (\frac{1}{\gam}\Pol)$ is a Gorenstein cone. 
See \Cref{fig:gorenstein} for an example.
In this paper we explore this definition for parameters $\gam=r$ and $\gam=2r$, other parameters are still to be investigated.  
The archetypal $r$-rational Gorenstein polytope is a rational polytope that contains the origin in its interior, see Corollary~\ref{cor:gorenstein}.
The definition of $\gamma$-rational Gorenstein does not require that the origin is contained in the polytope, hence, it does not require the existence of a polar dual.
A lattice polytope $\Pol$ is $1$-rational Gorenstein if and only if it is  a Gorenstein polytope in the classical sense.

Analogous to the lattice case, the following theorem shows that a polytope containing the origin is $r$-rational Gorenstein if and only if it has a palindromic $\rationalhstarpolynomial$-polynomial.
Let $\Pol = \{ \bx \in \RR^d : \, \bA \, \bx \le \bb \} $ be a rational
$d$-polytope, as in Equation~\eqref{eq:irredundantPol}. 
We may assume that there is an index $0 \le i \le n$ such that $b_j=0$ for $j=1,\dots,i$ and
$b_j\neq0$ for $j=i+1,\dots,n$; thus we can write $\Pol$ as follows: 
\begin{equation}\label{eq:splitPdescription}
\Pol \ = \ \left\{ \bx \in \RR^d : \
\begin{aligned}
 &\langle \ba_j , \bx  \rangle \le 0 &\text{ for }j =1,\dots,i 
\\
 & \langle \ba_j , \bx \rangle \le b_j &\text{ for } j =i+1,\dots,n
\end{aligned}
\right\} ,
\end{equation}
where $\ba_j$ are the rows of $\bA$.

\begin{theorem}\label{thm:gorensteinv2}
Let $\Pol = \{ \bx \in \RR^d : \, \bA \, \bx \le \bb\} $ be a rational $d$-polytope with codenominator $r$ and $\mathbf 0\in \Pol$, as in Equation~\eqref{eq:irredundantPol} and Equation~\eqref{eq:splitPdescription}.
Then the following are equivalent for $g,m \in \ZZ_{\geq 1}$ and $\frac{m}{r}\Pol$ a lattice polytope:
\begin{enumerate}[label=\textnormal{(\roman*)}]
\item \label{thm:gorenstein:gorenstein}
 $\Pol$ is $r$-rational Gorenstein with Gorenstein point $(g, \by) \in \hom (\frac{1}{r}\Pol)$.
\item \label{thm:gorenstein:eq}
There exists a (necessarily unique) integer solution $(g, \by)$ to
\begin{equation}\label{eq:gorenstein}
\begin{split}
 - \langle \ba_j ,\by \rangle &=1  \ \quad\text{ for }j=1,\dots,i\\
  b_j\,g-r\, \langle \ba_j,\by \rangle  &= b_j \quad \text{ for } j=i+1,\dots,n \,.
\end{split}
\end{equation}
\item \label{thm:gorenstein:palin}
$\rationalhstar{\Pol}{t}{m}$ is palindromic:
\[
  t^{(d+1)\frac{m}{r}-\frac g r} \, \rationalhstar{\Pol}{\frac 1 t}{m} \ = \ \rationalhstar{\Pol}{t}{m} \, .
\]
\item \label{thm:gorenstein:REhr}
$( -1)^{d+1}t^{\frac g r} \rationalEhr{\Pol}{t} = \rationalEhr{\Pol}{\frac 1 t}$.
\item \label{thm:gorenstein:rehr} $\rationalehr(\Pol;\frac n r)=\rationalehr(\Pol^\circ;\frac{n+g}{r})$ for all $n\in\ZZ_{\geq0}$.
\item \label{thm:gorenstein:dual}
$\hom ( \frac{1}{r} \Pol) ^\vee $ is the cone over a lattice polytope, i.e., there exists a lattice point $(g,\by)\in \hom(\frac 1 r \Pol)^\circ\cap\ZZ^{d+1}$ such that for every primitive ray generator $(v_0,\bv)$  
of $\hom(\frac 1 r \Pol)^\vee$
\begin{equation}
 \left\langle \left(g,\by\right), \left( v_0,\bv\right)
 \right\rangle = 1\,.
\end{equation}
\end{enumerate}

\end{theorem}
The equivalence of \ref{thm:gorenstein:gorenstein} and \ref{thm:gorenstein:dual} is well known (see, e.g., \cite[Definition 1.8]{BatyrevNill} or \cite[Exercises 2.13, 2.14]{brunsgubeladzektheory}); for the sake of completeness we include a proof below.

\begin{corollary}\label{cor:gorenstein}
Let $\Pol \subseteq \RR^d$ be a rational $d$-polytope with codenominator $r$.
If $\bzero \in \Pol^\circ$, then $\Pol$ is $r$-rational Gorenstein with
Gorenstein point $(1,0,\dots,0)$ and $\rationalhstar{\Pol}{t}{m}$ is palindromic.
\end{corollary}

\begin{example}[continued]
We check the Gorenstein criterion for the running examples such that $\mathbf 0 \in \Pol$.
\begin{enumerate}[(i)]
\item{
$\Pol_1 \coloneqq \left[-1, \frac 2 3\right]$, $r = 2$, $m=6$,
$$
\rationalhstar{\Pol_1}{t}{6} =
1+t^{\frac{1}{2}}+2t+3t^{\frac{3}{2}}+4t^2+4t^{\frac{5}{2}}+4t^3+4t^{\frac{7}{2}}+3t^4+2t^{\frac{9}{2}}+t^5+t^{\frac{11}{2}} .
$$
The polynomial $\rationalhstar{\Pol_1}{t}{6}$ is palindromic and therefore  (by Theorem \ref{thm:gorensteinv2}), $\Pol_1$ is $2$-rational Gorenstein.
This is to be expected; as $\mathbf{0} \in \Pol^{\circ}$,
\Cref{lem:hstarandtilde} shows that $\rationalhstar{\Pol_1}{t}{6}$ must be palindromic.
}
\item $\Pol_2 \coloneqq \left[0, \frac 2 3\right]$, $r$ = 2, $m$ = 3,
$$
\rationalhstar{\Pol_2}{t}{3} = 1+ t^{\frac 1 2} +t \, .
$$
The polynomial $\rationalhstar{\Pol_2}{t}{3}$ is palindromic and $\Pol_2$ is $2$-rational Gorenstein with Gorenstein point $(g,\by) = (4,1)$ $\in \hom(\frac{1}{2}\Pol_2)$.
\end{enumerate}
\end{example}

\begin{example}
The \emph{Haasenlieblingsdreieck} $\Delta \coloneqq \conv \{ (0,0), (2,0), (0,2) \}$ is not a Gorenstein polytope in the classic
(integral) setting, but it is 
$2$-rational Gorenstein: we compute
\[
  \rationalEhr{\Pol}{t}
  \ = \ \frac{ 1 }{ \left( 1 - t^{ \frac 1 2 } \right)^3 } 
  \ = \ \frac{ 1 + 3 t^{ \frac 1 2 } + 3t + t^{ \frac 3 2 } }{ \left(1 - t\right)^3 } \, .
\]
\end{example}

\begin{example}[A polytope that is not $\gamma$-rational Gorenstein for any $\gamma$]
 Let  $\nabla$ be the triangle as in \Cref{fig:nongorenstein}, i.e.,$\nabla= \conv \{ (0,0),(0,2),(5,2) \}$.
 Then the inequality description is
 \begin{equation}
 \nabla \ = \ \left\{(x_1,x_2)\in\RR^2\colon -x_1\leq 0\,,\ x_2\leq 2\,,\ 2x_1-5x_2 \leq 0\right\} .
 \end{equation}
We can read off the codenominator $r=2$ and compute its rational Ehrhart series with $m$ chosen minimally as
\begin{equation}
 \rationalEhr{\nabla}{t} \ = \ \frac{1 + 4 t^{\frac 1 2} + 7 t + 6 t^{\frac 3 2} + 2 t^2}{(1-t)^2}\,.
\end{equation}
 Hence, $\rationalhstar{\nabla}{t}{2}= 1 + 4 t^{\frac 1 2} + 7 t + 6 t^{\frac 3 2} + 2 t^2 $ is not palindromic and $\nabla$ is not rational Gorenstein.%
 \footnote{We thank Esme Bajo for suggesting this example and helping with computing it. See \cite{hstarboundary} for symmetric decompositions and boundary $\operatorname{h}^*$-polynomials.}
\end{example}

\begin{example}
The triangle $\nabla \coloneqq \conv \{ (0,0),(0,1),(3,1) \}$ has codenominator $1$. It is not $1$-rational Gorenstein as $\lvert \nabla^\circ \cap \ZZ^2 \rvert = 0$ and $\lvert (2\nabla)^\circ \cap \ZZ^2 \rvert = 2 $.
\end{example}
\begin{figure}
\centering
 \includegraphics[height=2.5cm]{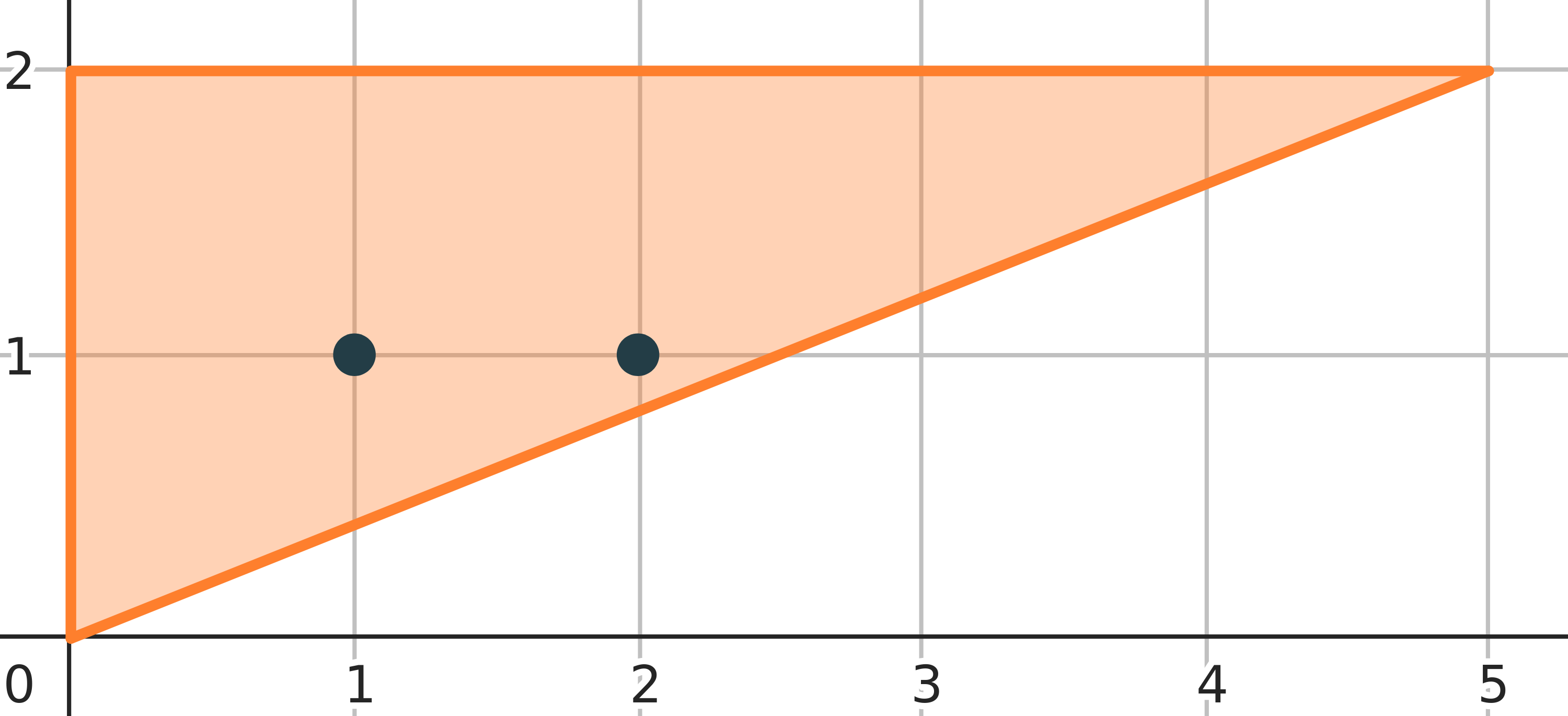}
 \caption{The triangle  $\nabla= \conv \{ (0,0),(0,2),(5,2) \}$, which is not rational Gorenstein.
 The cone $\hom(\frac{1}{\gamma}\nabla)$ contains two interior lattice points at lowest height, hence it does not posses a Gorenstein point.}\label{fig:nongorenstein}
\end{figure}

\begin{proof}[Proof of Theorem \ref{thm:gorensteinv2}]

 \begin{description}
  \item[\ref{thm:gorenstein:palin} $\Leftrightarrow$ \ref{thm:gorenstein:REhr} $\Leftrightarrow$ \ref{thm:gorenstein:rehr}] 
 We compute using reciprocity (see \Cref{cor:reciprocityseries}):
\begin{align}
 1+\sum_{\lambda\in\tfrac1 r \ZZ_{>0}}\rationalehr\mleft(\Pol;\lambda\mright)t^\lambda
\ &= \ \frac{\rationalhstar{\Pol}{t}{m}}{\left(1-t^{\frac m r}\right)^{\left(d+1\right)}}
\ = \ \frac{t^{\left(d+1\right)\frac{m}{r}-\frac g r} \, \rationalhstar{\Pol}{\frac 1 t}{m} }{\left(1-t^{\frac m r}\right)^{\left(d+1\right)}}\\
&= \ t^{-\frac{g}{r}}\frac{ \rationalhstar{\Pol^\circ}{ t}{m} }{\left(1-{t^{\frac m r}}\right)^{\left(d+1\right)}}
\ = \ t^{-\frac{g}{r}}\sum_{\lambda\in\tfrac1 r \ZZ_{>0}}\rationalehr\mleft(\Pol^\circ;\lambda\mright)t^\lambda\,.
\end{align}
That is equivalent to
\begin{align}
 t^{\frac{g}{r}}\rationalEhr{\Pol}{t}
 \ &= \ t^{\frac{g}{r}}\left( 1+\sum_{\lambda\in\tfrac1 r \ZZ_{>0}}\rationalehr\mleft(\Pol;\lambda\mright)t^\lambda\right)
 \ = \sum_{\lambda\in\tfrac1 r \ZZ_{>0}}\rationalehr\mleft(\Pol^\circ;\lambda\mright)t^\lambda \\
 &= \ \rationalEhr{\Pol^\circ}{t}
 \ = \ (-1)^{d+1}\rationalEhr{\Pol}{\frac 1 t}\,.
\end{align}
Comparing coefficients gives the third equivalence:
\begin{equation}
     \rationalehr\mleft(\Pol;\frac n r\mright)
     \ = \ \rationalehr\mleft(\Pol; \frac{n+g}{r}\mright) \quad \text{ for } n\in\ZZ_{\geq0}\,.
  \end{equation}

  \item[\ref{thm:gorenstein:rehr} $\Rightarrow$ \ref{thm:gorenstein:gorenstein}] 
 Since
 \begin{equation}
   \rationalehr\mleft(\Pol;\frac n r\mright)
     \ = \ \rationalehr\mleft(\Pol; \frac{n+g}{r}\mright) \quad \text{ for } n\in\ZZ_{\geq0}
 \end{equation}
 it suffices to show one inclusion:
 \begin{equation}
  \hom\mleft(\frac 1 r \Pol\mright)^\circ\cap\ZZ^{d+1} \ \supseteq \ \left(\left(g,\by\right) + \hom\mleft(\frac 1 r \Pol\mright)\right)\cap\ZZ^{d+1}\,,
 \end{equation}
  where $\by$ is the unique interior lattice point in $\frac{g}{r}\Pol^\circ$.
  Indeed,  for a point $(g,\by)\in \hom(\frac 1 r \Pol)^\circ\cap\ZZ^{d+1}$ it follows that 
 $(g,\by)+\bz\in \hom(\frac 1 r \Pol)^\circ\cap\ZZ^{d+1}$ for all 
 ${\bz\in\hom(\frac 1 r \Pol)\cap\ZZ^{d+1} }$.

  \item[\ref{thm:gorenstein:gorenstein} $\Rightarrow$ \ref{thm:gorenstein:palin}]
  By the definition of $\Pol$ being $r$-rational Gorenstein,
  \begin{equation}
   \hom\mleft(\frac 1 r \Pol \mright)^\circ \cap \ZZ^{d+1} \ = \ (g,\by)+\hom\mleft(\frac 1 r \Pol \mright) \cap \ZZ^{d+1}.
  \end{equation}
Computing integer point transforms gives:
\begin{equation}
 \sigma_{\hom \mleft( \frac 1 r \Pol \mright)^\circ}\left(\bz\right) \ = \ \bz^{(g,\by)} \sigma_{\hom \mleft( \frac 1 r \Pol \mright)}\left(\bz\right) .
\end{equation}
Applying reciprocity (see, e.g., \cite[Theorem~$4.3$]{ccd}) yields
\begin{equation}\label{eq:integerpointtransforms}
 \sigma_{\hom \mleft( \frac 1 r \Pol \mright)^\circ}\left(\bz\right)
 \ = \ \left(-1\right)^{d+1} \sigma_{\hom \mleft( \frac 1 r \Pol \mright)}\left(\frac 1 \bz\right)
 \ = \ \bz^{(g,\by)} \sigma_{\hom \mleft( \frac 1 r \Pol \mright)}\left(\bz\right) .
\end{equation}
By specializing $\bz=(t^{\frac 1 r}, 1,\dots,1)$ in Equation~\eqref{eq:integerpointtransforms} we obtain the following relation between Ehrhart series for $\frac 1 r \Pol$ in the variable $t^{\frac 1 r}$ and $t^{-\frac 1 r}$:
\begin{equation}\label{eq:aux}
\left(-1\right)^{d+1} \Ehr\mleft(\frac 1 r \Pol, \frac{1}{t^{\frac 1 r}}\mright)
\ = \ t^{\frac g r} \Ehr\mleft(\frac 1 r \Pol, t^{\frac 1 r}\mright) \, .
\end{equation}
From (the proof of) \Cref{thm:codem} we know that
\begin{equation}
 \Ehr\mleft(\frac 1 r \Pol, t^{\frac 1 r}\mright) \ = \ \rationalEhr{\Pol}{t} \ = \ \frac{\rationalhstar{\Pol}{t}{m}}{ \left( 1-t^{ \frac m r } \right)^{d+1}},
\end{equation}
where $m$ is an integer such that $\frac 1 r \Pol$ is a lattice polytope.
Substituting this into Equation~\eqref{eq:aux} yields
\begin{equation}
 \left(t^{ \frac m r }\right) ^{d+1}\frac{\rationalhstar{\Pol}{\frac 1 t}{m}}{
\left( 1-{t^{ \frac m r } }\right)^{d+1}} \ = \
\left(-1\right)^{d+1}\frac{\rationalhstar{\Pol}{ \frac 1 t}{m}}{ \left(
1-\frac{1}{t^{ \frac m r } }\right)^{d+1}} \ = \ t^{\frac g r}  \frac{\rationalhstar{\Pol}{t}{m}}{ \left( 1-t^{ \frac m r } \right)^{d+1}}
  \end{equation}
and thus
\begin{equation}
 t^{ \frac{\left(d+1\right)m} r -\frac g r }\, \rationalhstar{\Pol}{\frac 1
t}{m} \ = \ \rationalhstar{\Pol}{t}{m} \, .
\end{equation}
  
\item[\ref{thm:gorenstein:eq}  $\Leftrightarrow$ \ref{thm:gorenstein:dual}] 
The primitive ray generators of $\hom(\frac 1 r \Pol)^\vee$ are the primitive facet normals of $ \hom(\frac 1 r \Pol)$, that is,
\begin{equation}
 \left(0,-\ba_j\right)  \text{ for }j =1,\dots,i \quad\text{ and }\quad
 \left(1,-\frac{r}{b_j} \ba_j \right) \text{ for } j =i+1,\dots,n\,.
\end{equation}
Note that, since $\bzero\in\Pol$, $b_j\geq 0$ for all $j=1,\dots,n$.
The statement follows.
 \end{description}
 \begin{description}
  \item[\ref{thm:gorenstein:dual}  $\Rightarrow$ \ref{thm:gorenstein:gorenstein}]
  Since $(g,\by)\in\hom(\tfrac 1 r \Pol)^\circ\cap\ZZ^{d+1}$ is an interior point, $(g,\by)+\hom(\tfrac 1 r \Pol)\subseteq\hom(\tfrac 1 r \Pol)^\circ$ follows directly. 
  Let $(x_0,\bx)\in \hom(\tfrac 1 r \Pol)^\circ$, then for any primitive ray generator $(v_0,\bv)$ of $\hom(\tfrac 1 r \Pol)^\vee$ (being the primitive facet normals of $\hom(\tfrac 1 r \Pol)$),
  \begin{align}
   \left\langle \left(x_0,\bx\right)-\left(g,\by\right), \left(v_0,\bv\right)\right\rangle
   \ = \ \underbrace{\left\langle \left(x_0,\bx\right), \left(v_0,\bv\right)\right\rangle}_{>0}
	-\underbrace{\left\langle\left(g,\by\right), \left(v_0,\bv\right)\right\rangle}_{=1}
   \ \geq \ 0\,.
  \end{align}
  Hence, $(x_0,\bx)-(g,\by)\in\hom(\frac 1 r \Pol)$ and  
  $(x_0,\bx)\in(g,\by)+\hom(\frac 1 r \Pol)$. 
  \item[\ref{thm:gorenstein:gorenstein}  $\Rightarrow$ \ref{thm:gorenstein:dual}] 
  From the definition of Gorenstein point we know that $(g,\by)\in\hom(\frac 1 r \Pol)^\circ$ and hence
  \begin{equation}
   \left\langle \left(g,\by\right), \left(v_0,\bv\right)\right\rangle>0
  \end{equation}
  for all primitive facet normals $(v_0,\bv) $ of $\hom(\frac 1 r \Pol)$. 
  Since the facet normals $(v_0,\bv) $ are primitive, i.e., $\gcd((v_0,\bv) )=1$, there exists an integer point in the shifted hyperplane $H$ defined by
  \begin{equation}
   H \ = \ \left\{(x_0,\bx)\in\RR^{d+1} \colon\, \left\langle (v_0,\bv), (x_0,\bx)\right\rangle =1\right\}
  \end{equation}
  and hence $H$ contains a $d$-dimensional sublattice. Since the intersection  $H\cap\hom(\frac 1 r \Pol)^\circ$ \linebreak contains a pointed cone (e.g., the shifted recession cone), it contains a lattice point \linebreak ${(z_0,\bz)\in\hom(\frac 1 r \Pol)^\circ}$.
  
  So, for any facet of $\hom\left(\frac 1 r \Pol\right)$ there exists a lattice point $(z_0,\bz)$ in the interior of $\hom(\frac 1 r \Pol)$ at lattice distance one from the facet. 
  Since $(g,\by)+\hom(\frac 1 r \Pol)=\hom(\frac 1 r \Pol)^\circ$, there exists a point $(r_0,\br)\in\hom(\frac 1 r \Pol)$ such that
  \begin{equation}
    \left(g,\by\right)+\left(r_0,\br\right) \ = \ \left(z_0,\bz\right) .
  \end{equation}
  Then,
  \begin{align}
   1 \ = \ \left\langle\left(z_0,\bz\right),\left(v_0,\bv\right)\right\rangle
   \ = \ \underbrace{\left\langle\left(g,\by\right),\left(v_0,\bv\right)\right\rangle }_{>0}
   + \underbrace{\left\langle\left(r_0,\br\right),\left(v_0,\bv\right)\right\rangle}_{\geq0}
  \end{align}
  and $ \left\langle\left(g,\by\right),\left(v_0,\bv\right)\right\rangle =1$. \qedhere
 \end{description}
\end{proof}

As usual we state a version of \Cref{thm:gorensteinv2} for the refined rational Ehrhart series and the $\refinedrationalhstarpolynomial$-polynomial.
Here, the polytopes under consideration are not required to contain the origin.
This means that in the description of the polytope as in Equation~\eqref{eq:splitPdescription}  the vector $\bb\in\ZZ^n$ might have negative entries and we use absolute values when multiplying inequalities or facet normals with entries of $\bb$.
Except for this small difference, the proof is the same as that of \Cref{thm:gorensteinv2} so we omit it.

\begin{theorem}\label{thm:gorenstein2r}
Let $\Pol = \{ \bx \in \RR^d : \, \bA \, \bx \le \bb \} $ be a rational
$d$-polytope with codenominator $r$, as in Equation~\eqref{eq:irredundantPol} and Equation~\eqref{eq:splitPdescription}.
Then the following are equivalent for $g,m \in \ZZ_{\geq 1}$ and $\frac{m}{2r}\Pol$ a lattice polytope:
\begin{enumerate}[label=\textnormal{(\roman*)}]
\item \label{thm:gorenstein2r:gorenstein}
 $\Pol$ is $2r$-rational Gorenstein with Gorenstein point $(g, \by) \in \hom (\frac{1}{2r}\Pol)$.
\item \label{thm:gorenstein2r:eq}
There exists a (necessarily unique) integer solution $(g, \by)$ 
\begin{equation}\label{eq:gorenstein}
\begin{split}
   - \langle \ba_j ,\by \rangle &=1 \quad\text{ for }j=1,\dots,i\\
  b_j\,g-2r\, \langle \ba_j,\by \rangle  &= \lvert b_j\rvert \quad \text{ for } j=i+1,\dots,n \,.
\end{split}
\end{equation}
\item \label{thm:gorenstein2r:palin}
$\refinedrationalhstar{\Pol}{t}{m}$ is palindromic:
\[
  t^{(d+1)\frac{m}{2r}-\frac{ g}{2 r}} \, \refinedrationalhstar{\Pol}{\frac 1 t}{m} \ = \ \refinedrationalhstar{\Pol}{t}{m}\, .
\]
\item \label{thm:gorenstein2r:REhr}
$( -1)^{d+1}t^{\frac{ g}{2 r}} \refinedrationalEhr{\Pol}{t}
= \refinedrationalEhr{\Pol}{\frac{1}{t}}$.
\item \label{thm:gorenstein2r:rehr} $\rationalehr(\Pol;\frac{n}{2 r})=\rationalehr(\Pol^\circ;\frac{n+g}{2r})$ for all $n\in\ZZ_{\geq0}$.
\item \label{thm:gorenstein2r:dual}
$\hom ( \frac{1}{2r} \Pol) ^\vee $ is the cone over a lattice polytope, i.e., there exists a lattice point $(g,\by)\in \hom(\frac{1}{2 r} \Pol)^\circ\cap\ZZ^{d+1}$ such that for every primitive ray generator $(v_0,\bv)$ 
of $\hom(\frac{1}{2 r} \Pol)^\vee$
\begin{equation}
 \left\langle \left(g,\by\right), \left( v_0,\bv\right)
 \right\rangle = 1\,.
\end{equation}

\end{enumerate}
\end{theorem}
\Cref{thm:gorenstein2r} could be generalized to $\ell r$-rational Gorenstein polytopes for $\ell\in\ZZ_{>0}$.
However it is not clear that computationally this would provide any new insights to the (rational) Ehrhart theory of the polytopes.

\begin{corollary}
\hspace{1ex}
 \begin{enumerate}[label=\textnormal{(\roman*)}]
  \item If $\bzero\in\Pol^\circ$, then $\Pol$ is also $2r$-rational Gorenstein with the same Gorenstein point $(1,0\dots,0)$ (see \Cref{cor:gorenstein}).
  \item \label{rrationalisrrrational} If $\bzero\in\Pol$ and $\Pol$ is $r$-rational Gorenstein, then $\Pol$ is also $2r$-rational Gorenstein.
  \item If $\Pol$ is $2r$-rational Gorenstein and the first coordinate $g$ of the Gorenstein point $(g,\by)$ is even, then $\Pol$ is also $r$-rational Gorenstein.
 \end{enumerate}
\end{corollary}

\begin{proof}[Proof of \ref{rrationalisrrrational}]
	  Since $\bzero\in\Pol$ we know that $\rationalehr$ is constant on $[\tfrac{n}{r},\tfrac{n+1}{r})$ and we compute
\begingroup
\allowdisplaybreaks
  \begin{align}
         \refinedrationalEhr{\Pol}{t}
         \ &= \ 1+\sum_{n\in\ZZ_{>0}} \rationalehr\mleft(\Pol;\frac{ n}{2 r}\mright) t^{\frac{ n}{2 r}}\\
         &= \ 1+ \rationalehr\mleft(\Pol,\frac{1}{2r}\mright)t^{\frac{1}{2r}} \\*
         &\qquad+\sum_{n\in\ZZ_{>0}} \Bigg( \rationalehr\mleft(\Pol;\frac{2 n}{2r}\mright) t^{\frac{2 n}{2 r}} +   \underbrace{\rationalehr\mleft(\Pol;\frac{ 2n+1}{2r}\mright)}_{=\rationalehr\mleft(\Pol;\frac{ n}{r}\mright)} t^{\frac{ 2n+1}{2 r}}\Bigg)\\
         &= \ 1+t^{\frac{1}{2r}} +\sum_{n\in\ZZ_{>0}} \rationalehr\mleft(\Pol;\frac{ n}{ r}\mright) t^{\frac{ n}{ r}}\left(1+t^{\frac{1}{2r}}\right)\\
         &= \ \left(1+t^{\frac{1}{2r}}\right) \rationalEhr{\Pol}{t}\,,
        \end{align}
\endgroup
	where we also use that $\rationalehr\mleft(\Pol,0\mright)=\rationalehr\mleft(\Pol,\frac{1}{2r}\mright) =1$.
\end{proof}

\begin{SCfigure}[3][b]
\centering
 \includegraphics[height=7cm]{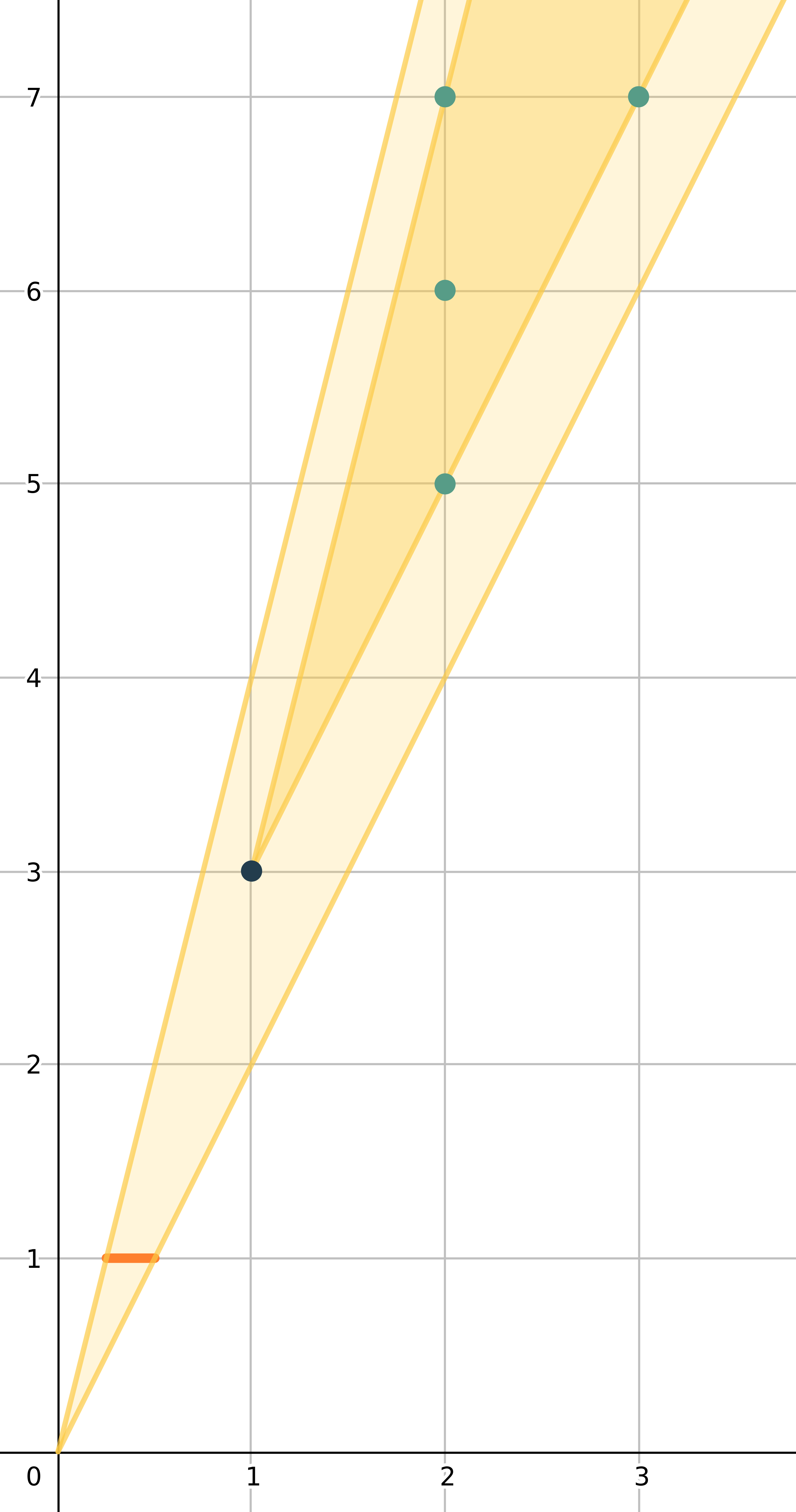}
 \caption{The  cone $\hom(\frac 1 4 \Pol_3)=\hom(\frac 1 8 \Pol_4)$ with Gorenstein point $(3,1)$ highlighted in dark blue.
 The other lattice points $\hom(\frac 1 4 \Pol_3)^\circ\cap\ZZ^2 $ are marked in blue.
 Observe that $(3,1)+\hom(\frac 1 4 \Pol_3) \cap \ZZ^2=\hom(\frac 1 4 \Pol_3)^\circ\cap\ZZ^2$. 
 }\label{fig:gorenstein}
\end{SCfigure}

\begin{example}(continued)
We check the Gorenstein criterion for the running examples such that $\mathbf 0 \notin \Pol$. 
\begin{enumerate}[(i)]
\setcounter{enumi}{2}
\item $\Pol_3 \coloneqq [1, 2]$, $r=2$, $m=4$,
$
\refinedrationalhstar{\Pol_3}{t}{4} = 1+t^{\frac 2 4}+t^{\frac 3 4}+t^{\frac 5 4}
$.
 \item 
 $\Pol_4 \coloneqq [2,4]$, $r=4$, 
  $m=4$,
 $
 \refinedrationalhstar{\Pol_4}{t}{4} = 1 + t^{\frac 1 4} +t^{\frac 3 8} + t^{\frac5 8}
$.
\end{enumerate}
Both polynomials $\refinedrationalhstar{\Pol_4}{t}{4}$ and $\refinedrationalhstar{\Pol_3}{t}{4}$ are palindromic and therefore $\Pol_3$ is $4$-rational Gorenstein and $\Pol_4$ is $8$-rational Gorenstein. 
In fact, $\tfrac{1}{4} \Pol_3 = \tfrac{1}{8} \Pol_4$ and so $\hom( \tfrac{1}{4} \Pol_3) = \hom (\tfrac{1}{8} \Pol_4)$. 
The Gorenstein point is $(g,\by) = (3,1)$.
\end{example}
%
%

\begin{example}[A polytope that is not $2r$-rational Gorenstein]\label{example:notgorenstein}
Let $\Pol_5 = [1,4]$. 
Then $r=4$ and $2r = 8$, so $\tfrac{1}{2r}\Pol_5 = [\tfrac 1 8 , \tfrac 1 2]$.
The first lattice point in the interior of $\hom(\tfrac{1}{8}\Pol_5) $ is $(g,\by) = (3,1)$.
However, $(3,1)$ does not satisfy Condition \ref{thm:gorenstein:eq} from Theorem \ref{thm:gorensteinv2}; it is at lattice distance 5 from one of the facets of $\hom( \tfrac{1}{8}\Pol_5)$.
\end{example}

\begin{remark}
Bajo and Beck \cite[Section 5]{hstarboundary} essentially showed that the $\hstar$-poly\-nomial
of a rational polytope $\Pol$ is palindromic if and only if $\hom(\Pol)$ is a Gorenstein cone. 
Hence, polytopes with palindromic $\hstar$-poly\-nomials, $\rationalhstarpolynomial$-poly\-nomials, or $\refinedrationalhstarpolynomial$-poly\-nomials are fully classified.
This implies, in particular, that polytopes with   palindromic $\hstar$-polynomials also have palindromic $\rationalhstarpolynomial$ and $\refinedrationalhstarpolynomial$-polynomials. 
%
 
\end{remark}


\section{Symmetric Decompositions}\label{sec:symmdecomp}

We now use the stipulations of the last section to give a new proof of the following
theorem. As we will see, our proof will also yield a rational version
(Theorem~\ref{thm:rationalbetkemcmullen} below).
\begin{theorem}[Betke--McMullen~\cite{betkemcmullen}]\label{thm:betkemcmullen}
Let $\Pol \subseteq \RR^d$ be a lattice $d$-polytope that contains 
a lattice point in its interior.
Then there exist polynomials $a(t)$ and $b(t)$ with nonnegative coefficients such that
\[
  \hstar\mleft(\Pol; t\mright) = a\left(t\right) + t \, b\left(t\right) \, ,
  \qquad
  t^d \, a\left( \tfrac 1 t\right) = a\left(t\right) \, ,
  \qquad
  t^{d-1} \, b\left( \tfrac 1 t\right) = b\left(t\right) \, .
\]
\end{theorem}

\begin{proof}
Suppose $\Pol$ is a lattice $d$-polytope with codenominator $r$.
If $\Pol$ contains 
a lattice point in its interior, we might as well assume it is the origin (the $\hstar$-polynomial \emph{is} invariant under lattice translations).
Then Corollary~\ref{cor:gorenstein} says
\begin{equation}\label{eq:symmgorenstein}
  t^{d+1-\frac 1 r} \, \rationalhstar{\Pol}{\frac 1 t}{r} \ = \ \rationalhstar{\Pol}{t}{r} \, .
\end{equation}
Note, since $\Pol$ is a lattice polytope we can choose $m=r$.
On the other hand, as we noted in the beginning of \Cref{sec:setup}, the $\hstar$-polynomial of a rational $d$-polytope always has a factor,
that carries over (by the proof of \Cref{thm:codem}) to
\[
  \rationalhstar{\Pol}{t}{r} \ =  \ \left( 1+t^{\frac 1 r}+\dots+t^{\frac{r-1}{r}} \right) \htilde\mleft(\Pol; t\mright)
\]
for some $\htilde(\Pol;t) \in \ZZ[t^{1/r}]$ (which is, of course, very much related to
Section~\ref{sec:stapledon}).
Moreover, by Equation~\eqref{eq:symmgorenstein} this polynomial satisfies $t^d \,
\htilde(\Pol;\frac 1 t) = \htilde(\Pol;t)$.
Note that
\begin{equation}\label{eq:betkemcmsetup}
  \rationalEhr{\Pol}{t}
  \ = \ \frac{ \left( 1+t^{\frac 1 r}+\dots+t^{\frac{r-1}{r}} \right) \htilde\mleft(\Pol; t\mright) }{ \left(1-t\right)^{ d+1 } }   
  \ = \ \frac{ \htilde\mleft(\Pol; t\mright) }{ \left(1-t^{ \frac 1 r } \right) \left(1-t\right)^d }   
\end{equation}
and the Gorenstein property of $\frac 1 r \Pol$ imply that $\htilde(\Pol;t)$ equals the
$h^*$-polynomial (in the variable $t^{ \frac 1 r }$) of the boundary of $\frac 1 r \Pol$.
Indeed, the rational Ehrhart series of $\partial\Pol$ is
\begin{align}
\rationalEhr{\Pol}{t}-\rationalEhr{\Pol^\circ}{t}
 \ &= \ \frac{\rationalhstar{\Pol}{t}{r}}{\left( 1-t\right)^{d+1}} -\frac{t^{d+1}\rationalhstar{\Pol}{\frac 1 t}{r}}{\left( 1-t\right)^{d+1}}\\
 &= \ \frac{\rationalhstar{\Pol}{t}{r}}{\left( 1-t\right)^{d+1}} -\frac{t^{\frac 1 r}\rationalhstar{\Pol}{t}{r}}{\left( 1-t\right)^{d+1}}\\
 &= \ \frac{(1-t^{\frac 1 r})\rationalhstar{\Pol}{t}{r}}{\left( 1-t\right)^{d+1}} 
 \ = \ \frac{\htilde\mleft(\Pol;t\mright)}{\left( 1-t\right)^{d}}\,.
\end{align}
The (triangulated) boundary of a polytope is shellable \cite[Chapter 8]{ziegler}, and this
shelling gives a half-open decomposition of the boundary, which yields nonnegativity of 
the $\hstar$-vector. Hence,~$\htilde(\Pol;t)$ has nonnegative coefficients.

Recall that $\Int$ is the operator that extracts from a polynomial in $\ZZ[t^{ \frac 1 r }]$ the
terms with integer powers of~$t$.
Thus
\[
  a\mleft(t\mright) \ \coloneqq \ \Int \mleft( \htilde\mleft(\Pol;t\mright) \mright)
\]
is a polynomial in $\ZZ[t]$ with nonnegative coefficients satisfying $t^d \, a( \tfrac 1 t) = a(t)$.
(Note that $a(t)$ can be interpreted as the $h^*$-polynomial of the boundary of $\Pol$;
see, e.g.,~\cite{hstarboundary}.)
Again, because we could choose $m=r$, we compute using Equation~\eqref{eq:betkemcmsetup}:
\begin{align}
  \hstar\mleft(\Pol; t\mright)
  \ &= \ \Int \mleft( \left( 1 + t^{ \frac 1 r } + \dots + t^{ \frac{ r-1 }{ r } } \right) \htilde\mleft(\Pol;t\mright) \mright) \\
  \ &= \ a\mleft(t\mright) + \Int \mleft( \left( t^{ \frac 1 r } + t^{ \frac 2 r } + \dots + t^{ \frac{ r-1 }{ r } } \right) \htilde\mleft(\Pol;t\mright) \mright) .
\end{align}
Since $\beta\mleft(t\mright) \coloneqq \left( t^{ \frac 1 r } + t^{ \frac 2 r } + \dots + t^{ \frac{ r-1 }{ r }
} \right) \htilde\mleft(\Pol;t\mright)$ satisfies $t^{ d+1 } \, \beta\mleft(\frac 1 t\mright) = \beta\mleft(t\mright)$, the
polynomial
\[
  b\mleft(t\mright) \ \coloneqq \ \frac 1 t \Int \mleft( \left( t^{ \frac 1 r } + t^{ \frac 2 r } + \dots + t^{ \frac{ r-1 }{ r } } \right) \htilde\mleft(\Pol;t\mright) \mright)
\]
satisfies $t^{d-1} \, b\mleft( \tfrac 1 t\mright) = b\mleft(t\mright)$, and $\hstar\mleft(\Pol; t\mright) = a\mleft(t\mright) + t \, b\mleft(t\mright)$ by
construction.
\end{proof}

\begin{remark} 
We could have started the proof of \Cref{thm:betkemcmullen} with Equation~\eqref{eq:stapledonsetup} and then used Stapledon's result \cite{stapledonfreesums} that $\htilde(\Pol;t)$ is palindromic and nonnegative.
\end{remark}

The rational version of this theorem is a special case of
\cite[Theorem~4.7]{beckbraunvindasmelendez}.

\begin{theorem}\label{thm:rationalbetkemcmullen}
Let $\Qol \subseteq \RR^d$ be a rational $d$-polytope with denominator $k$ that contains
a lattice point in its interior.
Then there exist polynomials $a(t)$ and $b(t)$ with nonnegative coefficients such that
\[
  \hstar\mleft(\Qol; t\mright) = a\mleft(t\mright) + t \, b\mleft(t\mright) \, ,
  \qquad
  t^{k\left(d+1\right) - 1} \, a\mleft( \tfrac 1 t\mright) = a\left(t\right) \, ,
  \qquad
  t^{k\left(d+1\right) - 2} \, b\mleft( \tfrac 1 t\mright) = b\mleft(t\mright) \, .
\]
\end{theorem}

\begin{proof}
We repeat our proof of Theorem~\ref{thm:betkemcmullen} for $\Pol \coloneqq k \, \Qol$, except that
instead of the operator $\Int$, we use the operator $\Rat_k$ which extracts the terms with
powers that are multiples of $\frac 1 k$.
So now
\begin{equation}
\begin{split}
a(t) &\ \coloneqq \ \Rat_k ( \htilde(\Pol;t) ) , \\
b\mleft(t\mright) &\ \coloneqq \ \frac 1 {t^{ \frac 1 k }} \Rat_k \mleft( \left( t^{ \frac 1 r } + t^{ \frac 2 r } + \dots + t^{ \frac{ r-1 }{ r } } \right) \htilde\mleft(\Pol;t\mright) \mright)
  , \text{ and }\\
  \hstar(\Pol; t) &\ = \ a(t^k) + t \, b(t^k) \, . \qedhere
  \end{split}
  \end{equation}
\end{proof}

We conclude by remarking that there is a generalization of the Betke--McMullen theorem
due to Stapledon~\cite{stapledondelta}; here the assumption of an interior lattice point is
dropped, but the symmetric decomposition happens now with a modified $\hstar$-polynomial.
A rational version is the afore-mentioned~\cite[Theorem~4.7]{beckbraunvindasmelendez}; see also~\cite{hstarboundary}.


\section{Period Collapse}\label{sec:periodcollapse}

One of the classic instances of period collapse in integral Ehrhart theory is the triangle
\begin{equation}\label{ex:ptriangle}
 \Delta \ \coloneqq \ \conv \{ (0,0), (1, \tfrac{ p-1 }{ p }), (p,0) \}
\end{equation}
 where $p \ge 2$ is an integer~\cite{mcallisterwoods}; see also~\cite{cristofarogardiner}
for an irrational version. 
Here 
\begin{equation}\label{}
  \Ehr\mleft(\Delta;t\mright) \ = \ \frac{ 1+\left(p-2\right)t }{ \left(1-t\right)^3 } 
\end{equation}
and so, while the denominator of $\Delta$ equals $p$, the period of
$\ehr(\Delta; n)$ collapses to 1: the quasipolynomial
$\ehr(\Delta; n) = \frac{ p-1 }{ 2 } \, n^2 + \frac{ p+1 }{ 2 } \, n + 1$ is a
polynomial.

As mentioned in the Introduction, we offer data points towards the
question of whether or how much period collapse happens in
rational Ehrhart theory, and how it compares to the classical scenario.

\begin{example}
We consider the triangle $\Delta$ defined in Equation~\eqref{ex:ptriangle} with $p=3$.
Note that both denominator and codenominator of $\Delta$ equal~$3$. 
We compute 
\[
  \rationalEhr{\Delta}{t} \ = \ \frac{ 1 + t^{ \frac 5 3 } }{ \left( 1 - t^{ \frac 1 3 } \right)^2 \left( 1-t^3 \right) } \, .
\]
Note that the accompanying rational Ehrhart quasipolynomial $\rationalehr(\Pol; \lambda)$ thus has period~3.
We can retrieve the integral Ehrhart series from the rational by rewriting
\[
  \rationalEhr{\Delta}{t}
  \ = \ \frac{ \left( 1 + t^{ \frac 5 3 } \right) \left( 1 + t^{ \frac 1 3 } + t^{ \frac 2 3 } \right)^2 }{ \left( 1 - t \right)^2 \left( 1-t^3 \right) } 
  \ = \ \frac{ \left( 1 + t^{ \frac 5 3 } \right) \left( 1 + 2 t^{ \frac 1 3 } + 3 t^{
\frac 2 3 } + 2 t + t^{ \frac 4 3 } \right) }{ \left( 1 - t \right)^2 \left( 1-t^3 \right) } 
\]
and then disregarding the fractional powers in the numerator, which gives
\[
  \Ehr\mleft(\Delta;t\mright) \ = \ \frac{ 1 + 2t + 2t^2 + t^3 }{ \left(1-t\right)^2 \left(1-t^3\right) }
  \ = \ \frac{ 1+t }{ \left(1-t\right)^3 } \, . 
\]
Hence the classical Ehrhart quasipolynomial exhibits period collapse while the rational does not.
\end{example}

\begin{example}
The recent paper \cite{fernandespinaalfonsinrobins} studied certain families
of polytopes arising from graphs, which exhibit period collapse. One example
is the pyramid
\[
  \Pol_5 \ \coloneqq \ \conv \left\{ \left(0,0,0\right), \left(\tfrac 1 2, 0, 0\right), \left(0, \tfrac 1 2, 0\right),
\left(\tfrac 1 2, \tfrac 1 2, 0\right), \left(\tfrac 1 4, \tfrac 1 4, \tfrac 1 2\right) \right\} .
\]
which has denominator 4 and codenominator 1.
In particular, its rational Ehrhart series equals the standard Ehrhart series,
and
\[
  \rationalEhr{\Pol_5}{t} \ = \ \Ehr\mleft(\Pol_5;t\mright) \ = \ \frac{ 1+t^3 }{ \left(1-t\right)
\left(1-t^2\right)^3 } 
\]
shows that $\ehr(\Pol_5; n)$ and $\rationalehr(\Pol_5; \lambda)$ both have period 2, i.e.,
they both exhibit period collapse.
\end{example}

\begin{example}
Recall the running examples $\Pol_1=[-1,\frac 2 3]$ and $\Pol_2=[0,\frac 2 3]$.
Restricting the rational Ehrhart quasipolynomial from page~\pageref{example:running} to positive integers we retrieve the Ehrhart quasipolynomials:
\begin{align}
 &\ehr\mleft(\Pol_1;n\mright) \ = \ \begin{cases}
                 \frac 5 3 n+1 \quad&\text{if } n \equiv 0 \bmod{3},\\
                 \frac 5 3 n+\frac 1 3 \quad&\text{if } n \equiv 1 \bmod{3},\\
                 \frac 5 3 n+\frac 2 3 \quad&\text{if } n \equiv 2 \bmod{3},
                \end{cases}\\
  &\ehr\mleft(\Pol_2;n\mright) \ = \ \begin{cases}
                 \frac 2 3 n+1 \quad&\text{if } n \equiv 0 \bmod{3},\\
                 \frac 2 3 n+\frac 1 3 \quad&\text{if } n \equiv 1 \bmod{3},\\
                 \frac 2 3 n+\frac 2 3 \quad&\text{if } n \equiv 2 \bmod{3}.
                \end{cases}
\end{align}
We can observe the period 3 here for both functions.
Recall the rational Ehrhart series from page~\pageref{example:running:series}:
\begin{align}
\rationalEhr{\Pol_1}{t} \ &= \ \frac{1 + t^{\frac{1}{2}}+t+t^{\frac{3}{2}}+t^2}{\left(1-t\right)\left(1-t^{\frac{3}{2}}\right)}\,,
\\
\rationalEhr{\Pol_2}{t} \ &= \ \frac{1}{\left(1-t^{\frac 1 2
}\right)\left(1-t^{\frac 3 2}\right)} = \frac{1+ t^{\frac 1 2}+ t}{\left(1-t^{\frac 3
2}\right)^2}\,.
\end{align}
We can read off from the series that $\rationalehr(\Pol_1;\lambda)$ has rational period
3, whereas $\frac 3 2$ is the rational period of $\rationalehr(\Pol_2;\lambda)$.
Both $\Pol_1$ and $\Pol_2$ have codenominator $r=2$, but $m_{\Pol_1}=6$ and
$m_{\Pol_2}=3$ (see computations on page~\pageref{example:running:series}). So the
expected period is $\frac 6 2 = 3$ for $\Pol_1$ and $\frac 3 2$ for $\Pol_2$. Thus here neither the rational Ehrhart quasipolynomials nor the integral Ehrhart
quasipolynomials exhibit period collapse.

\end{example}

We do not know any examples of polytopes with period collapse in their rational Ehrhart quasipolynomials but not in their integral Ehrhart quasipolynomials.
The question about possible period collapse of an Ehrhart quasipolynomial is only one of
many one can ask for a given rational polytope. For example, there
are many interesting questions and conjectures on when the $\hstar$-polynomial is unimodal. 
One can, naturally, extend any such question to rational Ehrhart series.
Finally, our results generalize to polynomial-weight counting functions of rational
polytopes (see, e.g.,~\cite{baldoniberlinekoeppevergnerealdilation}), where $\rationalehr(\Pol; \lambda)$ gets replaced by
$
  \sum_{ \bx \in \lambda \Pol \cap \ZZ^d } p(\bx)
$
for a fixed polynomial $p(\bx) \in \CC[x_1, \dots, x_d]$.

\section*{Acknowledgments}
We are grateful to an anonymous referee for numerous helpful comments and suggestions.

\bibliographystyle{amsplain}
\bibliography{bib}

\setlength{\parskip}{0cm} 
%

\end{document}